\documentclass{amsart}
\usepackage{psfrag}
\usepackage{color,xcolor}
\usepackage{tikz}
\usetikzlibrary{matrix,arrows}
\usepackage{graphicx,graphics}
\usepackage{amssymb,amsfonts,amsmath,amstext,amsthm,amscd,verbatim,enumerate,etoolbox,mathrsfs,hyperref}
\usepackage[T1]{fontenc}

\DeclareMathOperator{\dist}{dist}
\DeclareMathOperator{\Hull}{Hull}
\DeclareMathOperator{\card}{card}

\DeclareMathAlphabet{\mathpzc}{OT1}{pzc}{m}{it}

\begin{document}

\newtheorem{theorem}{Theorem}[section]
\newtheorem{result}[theorem]{Result}
\newtheorem{fact}[theorem]{Fact}
\newtheorem{conjecture}[theorem]{Conjecture}
\newtheorem{lemma}[theorem]{Lemma}
\newtheorem{proposition}[theorem]{Proposition}
\newtheorem{corollary}[theorem]{Corollary}
\newtheorem{facts}[theorem]{Facts}
\newtheorem{props}[theorem]{Properties}
\newtheorem*{thmA}{Theorem A}
\newtheorem{ex}[theorem]{Example}
\theoremstyle{definition}
\newtheorem{definition}[theorem]{Definition}
\newtheorem{remark}[theorem]{Remark}
\newtheorem{example}[theorem]{Example}
\newtheorem*{defna}{Definition}

\newcommand{\notes} {\noindent \textbf{Notes.  }}
\newcommand{\defn} {\noindent \textbf{Definition.  }}
\newcommand{\defns} {\noindent \textbf{Definitions.  }}
\newcommand{\x}{{\bf x}}
\newcommand{\e}{\epsilon}
\renewcommand{\d}{\delta}
\renewcommand{\a}{\alpha}
\newcommand{\z}{{\bf z}}
\newcommand{\B}{{\bf b}}
\newcommand{\V}{{\bf v}}
\newcommand{\T}{\mathbb{T}}
\newcommand{\Z}{\mathbb{Z}}
\newcommand{\Hp}{\mathbb{H}}
\newcommand{\D}{\Delta}
\newcommand{\g}{\gamma}
\newcommand{\G}{\Gamma}
\newcommand{\R}{\mathbb{R}}
\newcommand{\N}{\mathbb{N}}
\renewcommand{\B}{\mathbb{B}}
\renewcommand{\S}{\mathbb{S}}
\newcommand{\C}{\mathbb{C}}
\newcommand{\ddist}{\text{dist}^*}
\newcommand{\ft}{\widetilde{f}}
\newcommand{\dt}{{\mathrm{det }\;}}
 \newcommand{\adj}{{\mathrm{adj}\;}}
 \newcommand{\0}{{\bf O}}
 \newcommand{\av}{\arrowvert}
 \newcommand{\zbar}{\overline{z}}
 \newcommand{\xbar}{\overline{X}}
 \newcommand{\htt}{\widetilde{h}}
\newcommand{\ty}{\mathcal{T}}
\newcommand\diam{\operatorname{diam}}
\newcommand\Mod{\operatorname{Mod}}
\renewcommand\Re{\operatorname{Re}}
\renewcommand\Im{\operatorname{Im}}
\newcommand{\tr}{\operatorname{Tr}}
\renewcommand{\skew}{\operatorname{skew}}

\newcommand{\ds}{\displaystyle}
\numberwithin{equation}{section}

\renewcommand{\theenumi}{(\roman{enumi})}
\renewcommand{\labelenumi}{\theenumi}

\newcommand{\vyron}[1]{{\scriptsize \color{blue}\textbf{Vyron's note:} #1 \color{black}\normalsize}}
\newcommand{\alastair}[1]{{\scriptsize \color{red}\textbf{Alastair's note:} #1 \color{black}\normalsize}}

\setcounter{tocdepth}{1}

\title{Decomposing multitwists}

\author{Alastair N. Fletcher}
\email{fletcher@math.niu.edu}
\address{Department of Mathematical Sciences, Northern Illinois University, Dekalb, IL 60115, USA}

\author{Vyron Vellis}
\email{vvellis@utk.edu}
\address{Department of Mathematics, The University of Tennessee, Knoxville, TN 37966, USA}

\thanks{A.F. was supported by a grant from the Simons Foundation, \#352034. V.V. was partially supported by the NSF DMS grant 1952510.}
\subjclass[2020]{Primary 30L05; Secondary 30F45, 37C10}
\keywords{uniformly disconnected set, bi-Lipschitz decomposition, hyperbolic surface of infinite type, bi-Lipschitz path}

\maketitle

\begin{abstract}
The Decomposition Problem in the class $LIP(\S^2)$ is to decompose any bi-Lipschitz map $f:\S^2 \to \S^2$ as a composition of finitely many maps of arbitrarily small isometric distortion. In this paper, we construct a decomposition for certain bi-Lipschitz maps which spiral around every point of a Cantor set $X$ of Assouad dimension strictly smaller than one. These maps are constructed by considering a collection of Dehn twists on the Riemann surface $\S^2 \setminus X$. The decomposition is then obtained via a bi-Lipschitz path which simultaneously unwinds these Dehn twists. As part of our construction, we also show that $X \subset \S^2$ is uniformly disconnected if and only if the Riemann surface $\S^2 \setminus X$ has a pants decomposition whose cuffs have hyperbolic length uniformly bounded above, which may be of independent interest.
\end{abstract}

%\tableofcontents

\section{Introduction}

A bi-Lipschitz homeomorphism $f:X\to Y$ between metric spaces is a homeomorphism that roughly preserves absolute distances; specifically, there exists $L\geq 1$ such that
\[ L^{-1} d_X(x,y) \leq d_Y( f(x) , f(y) ) \leq Ld_X(x,y)\]
for all $x,y\in X$. We then say that $f$ is an $L$-bi-Lipschitz map. The smallest such constant $L$ is called the {\it isometric distortion} of $f$. Letting $\S^n$ be the sphere of dimension $n$, we denote by $LIP(\S^n)$ the class of orientation preserving homeomorphisms of $\S^n$. 

A central problem in bi-Lipschitz geometry is whether a bi-Lipschitz map can be decomposed into bi-Lipschitz mappings of arbitrarily small isometric distortion.

\begin{conjecture}[Decomposition Problem]
Let $n\geq 1$ and let $f\in LIP(\S^n)$. Then for every $\epsilon >0$ we can find homeomorphisms $f_k \in LIP(\S^n)$, for $k=1,\ldots, m$, such that $f$ can be written as a composition $f = f_m \circ \ldots \circ f_1$, where each $f_k$ has isometric distortion at most $1+\epsilon$.
\end{conjecture}

The case $n=1$ is elementary: suppose $I,J$ are intervals in $\R$ and $f:I\to J$ is an $L$-bi-Lipschitz map. Then $f$ can be written as $f = f_2\circ f_1$, where
\[ f_1(x) = \int_{x_0}^x |f'(t) |^{\lambda} \: dt,\]
$x_0$ is fixed, $\lambda = \log_L \alpha$, $f_1$ is $\alpha$-bi-Lipschitz and $f_2 = f\circ f_1^{-1}$ is $L/\alpha$-bi-Lipschitz.

However, for $n\geq 2$, the Decomposition Problem has been so far elusive. It is clear that affine bi-Lipschitz mappings can be factored into affine mappings of small isometric distortion, but beyond this, only certain specific examples have been considered. Freedman and He \cite{FrHe} studied the logarithmic spiral map $s_k(z) = ze^{ik \log |z| }$, which is an $L$-bi-Lipschitz map with $|k| = L-1/L$. Gutlyanskii and Martio \cite{GuMa} studied a related class of mappings in dimension $2$, and generalized this to a class of volume preserving bi-Lipschitz automorphisms of the unit ball $\B^3$ in three dimensions.

Although in this paper we focus on $LIP(\S^2)$, the Decompsition Problem can also be asked for the class of quasiconformal homeomorphisms of $\S^n$. In dimension $2$, the fact that every quasiconformal map arises as a solution of the Beltrami equation can be leveraged to show that the Decomposition Problem has a positive solution here; see \cite[Theorem 4.7]{Lehto}. Since every orientation preserving bi-Lipschitz map is also quasiconformal, in dimension $2$ we are able to find a decomposition of bi-Lipschitz maps, but only into quasiconformal maps of small conformal distortion. Observe, however, that quasiconformal maps need not be bi-Lipschitz.

A similar problem was studied by the first named author and Markovic in \cite{FM}. There it was shown that $C^1$ diffeomorphisms of $\S^n$, for $n\geq 2$, can be decomposed into bi-Lipschitz maps of arbitrarily small isometric distortion. This solves the Decomposition Problem for $C^1$ bi-Lipschitz maps, but of course, bi-Lipschitz maps are only guaranteed to be differentiable almost everywhere.

In this paper, we study the Decomposition Problem for a class of maps in $LIP(\S^2)$ which spiral around every point of a Cantor set, with small Assouad dimension; see below for definitions. Necessarily these maps are not differentiable at any point of the Cantor set in question. This can be viewed as a generalization of the result of Freedman and He, although they were motivated to give estimates on the number of maps required in the decomposition. Our constructions will be involved enough that we will not address this question here, and be content to just find a decomposition.

Maps which spiral around every point of a Cantor set simultaneously are not new. Such mappings  were constructed by Astala et al in \cite{AIPS} in order to give sharp examples of the multifractal spectrum; see in particular the proof of Theorem 5.1 and Figure 7 in \cite{AIPS}.

\subsection{Uniformly disconnected sets and hyperbolic geometry}

We identify the topological sphere $\S^2$ with the one point compactification $\R^2 \cup \{ \infty \}$, and equip it with the chordal metric. If $X \subset \S^2$ is a Cantor set, then by applying a chordal isometry we may assume that $X\subset \R^2$. Having done this, we may then view $S:= \S^2 \setminus X$ as a Riemann surface of infinite type. 

The bi-Lipschitz maps that we will decompose arise from a collection of Dehn twists on the surface $S$. For the mappings we define to be bi-Lipschitz, we need some control on the ring domains on which the Dehn twists are defined. Informally, these ring domains cannot be too thin, and their boundaries cannot be too wiggly.

To address the first of these points, we recall some hyperbolic geometry. The surface $S$ has a pants decomposition, that is, $S = \bigcup_{i=1}^{\infty} P_i$, where each $P_i$ is a topological sphere with three disks removed. 
%and the three boundary curves of $P_i$ are closed geodesics. 
The collection of boundary curves of the pairs of pants, called the \emph{cuffs} of the decomposition, may be enumerated by $(\alpha_j)_{j=1}^{\infty}$. 
Each $\alpha_j$ is a simple closed curve on $S$ and generates a class $[\alpha_j]$ of simple closed curves that are freely homotopic to $\alpha_j$.

We denote by $\ell_S (\alpha_j)$ the hyperbolic length of $\alpha_j$ and by $\ell_S[\alpha_j]$ the infimum of hyperbolic lengths of closed curves in $S$ homotopic to $\alpha_j$. We suppress the subscript $S$ if the context is clear. It is well-known \emph{a Cantor set $X\subset\R^2$ is uniformly perfect if and only if for any pants decomposition of $\mathbb{S}^2\setminus X$, the associated cuffs $(\alpha_j)_{j=1}^{\infty}$ satisfy $\inf_j \ell_S[ \alpha_j] >0$}; see \cite{Po79}.
Recall that a non-degenerate metric space $X$ is \emph{uniformly perfect} if there exists a constant $C>1$ such that for any $x\in X$ and every positive $r< \diam{X}$, we have that $B(x,r)\setminus B(x,r/C) \neq \emptyset$. Informally, this means that any ring domain separating $X$ cannot be too thick.

Uniform disconnectedness is, in a sense, the opposite of uniform perfectness; a metric space $X$ is \emph{uniformly disconnected} if there exists a constant $C>1$ such that for any $x\in X$ and every positive $r< \diam{X}$, there exists $X_{x,r} \subset X$ that contains $x$ such that $\diam{X_{x,r}}\leq r$ and 
\[ \dist(X_{x,r}, X\setminus X_{x,r})  \geq r/C.\]
It is natural to ask whether $X$ being uniformly disconnected implies analogous geometric properties of the surface $S$. Our first result gives such a characterization.

\begin{theorem}\label{thm:UD-hyper}
A Cantor set $X\subset \R^2$ is uniformly disconnected if and only if there exists a pants decomposition for $S=\mathbb{S}^2\setminus X$ such that the associated cuffs $(\alpha_j)_{j=1}^{\infty}$ satisfy $\sup_j \ell [ \alpha_j ] < \infty$.
\end{theorem}

By a uniformization theorem of David and Semmes \cite{DSbook}, a set $X \subset \R^2$ is quasisymmetrically homeomorphic to the standard ternary Cantor set $\mathcal{C}$ if and only if it is compact, uniformly perfect, and uniformly disconnected. Therefore, by Theorem \ref{thm:UD-hyper} and \cite{Po79}, it follows that a Cantor set $X\subset \R^2$ is quasisymmetrically homeomorphic to $\mathcal{C}$ if and only if there exists a constant $C>1$ and a pants decomposition for $S=\mathbb{S}^2\setminus X$ such that the associated cuffs $(\alpha_j)_{j=1}^{\infty}$ satisfy
\[ C^{-1} \leq \ell [ \alpha_j ] \leq C, \qquad\text{for all $j$}.\]

\subsection{Dehn multi-twists}

Here we outline how our bi-Lipschitz mappings are constructed. Full definitions and discussion will follow in the sequel. The first step is the following proposition which is a corollary of Theorem \ref{thm:UD-hyper}. 
%Here and for the rest of the paper, given $x\in \R^2$ and $0<r<R$, we denote $A(x,r,R) = \overline{B}(x,R)\setminus B(x,r)$

%Recall that a homeomorphism $f:X\to Y$ between metric spaces is an \emph{$L$-quasisimilarity} for some $L\geq 1$, if there exists $\lambda>0$ such that $\lambda^{-1}f$ is $L$-bi-Lipschitz. 

\begin{proposition}\label{prop:main}
Given $c\geq 1$, there exists $L>1$, $k\in\N$, and a finite set 
\[ \{g_i : \overline{B}(0,1)\setminus B(0,1-\tfrac1{L}) \to \R^2\}_{i=1}^{k}\] 
of $L$-bi-Lipschitz conformal maps with the following property. Let $X\subset \R^2$ be a $c$-uniformly disconnected Cantor set and let $(\alpha_j)_{j=1}^{\infty}$ be the cuffs from Theorem \ref{thm:UD-hyper}. There exist mutually disjoint closed ring domains $R_j \subset \R^2 \setminus X$ homotopic to $\a_j$, and similarities $(\phi_j)_{j=1}^{\infty}$ of $\R^2$ such that for each $j\in\N$ there exists $i(j)\in \{1,\dots,k\}$ with $R_j := \phi_j\circ g_{i(j)} (\overline{B}(0,1)\setminus B(0,1-\frac1{L}))$.  Moreover, for each $j\in\N$, the bounded component of $\phi_j^{-1}(R_j)$ has diameter equal to 1.
\end{proposition}

This proposition says that given a pants decomposition of $\S^2 \setminus X$, we can find a collection of rings on which our map $f$ will be supported with the property that, up to similarity, the rings are chosen from a finite set. This finiteness will lead to a certain uniformity in the Dehn twists that define $f$.

More precisely, fix a Cantor set $X \subset \R^2$ for which the Assouad dimension satisfies $\dim _A X <1$. It then follows from \cite{Lu} that $X$ is uniformly disconnected. Let $R_j$ and $f_j:= \phi_j\circ g_{i(j)}$ be the ring domains and conformal maps, respectively, from Proposition \ref{prop:main}. Then, a Dehn twist can be defined on each $\overline{R_j}$ by
\[ f|\overline{R_j} := f_j\circ \mathfrak{D} \circ f_j^{-1}\]
where $\mathfrak{D}: \overline{B}(0,1)\setminus B(0,1-\frac1{L}) \to \overline{B}(0,1)\setminus  B(0,1-\frac1{L})$ is the Dehn twist
\[ \mathfrak{D}(r,\theta) = (r,\theta + 2\pi L(1-r)).\]
Let $f:\R^2 \to \R^2$ be given by the Dehn twist in each $\overline{R_j}$ as above, and the identity elsewhere. The uniform bi-Lipschitz constant of maps $g_j$ guarantees that $f$ is a bi-Lipschitz map; see Lemma \ref{lem:BLmap}. The main theorem of this paper reads as follows:

\begin{theorem}
\label{thm:main1}
If $X \subset \R^2$ is a Cantor set with $\dim_A(X) <1$ and if $f$ is the bi-Lipschitz map defined above, then given $\epsilon >0$, there exists $N\in \N$ such that $f = f_N\circ \ldots \circ f_1$, where each $f_j$, for $j=1,\ldots, N$, is $(1+\epsilon)$-bi-Lipschitz.
\end{theorem}

It is worth pointing out that if the rings $R_j$ can initially be chosen to be round rings, such as those constructed in \cite{AIPS}, then the assumption $\dim_A(X)<1$ can be replaced by uniform disconnectedness, and we can decompose $f$ directly in this case. In fact, the assumption $\dim_A{X}<1$ can be dropped, see Section \ref{sec:example} for an example), and we conjecture that it can be replaced by uniform disconnectedness.

\subsection{Strategy of the proof}

The crux of the proof is to construct a bi-Lipschitz path from the identity to $f$. Bi-Lipschitz paths were introduced in \cite{FM} to provide a way to deform one bi-Lipschitz mapping to another in a controlled way. Partitioning the path into small subintervals yields the required decomposition.

Consider first the special case where each of the rings $R_j$ are round, see Figure \ref{fig:1}. Writing $V_j$ for the bounded component of the complement of $R_j$, we can unwind the Dehn twist supported in $R_j$ in the obvious way, and extend this unwinding via the identity in the unbounded component of the complement of $R_j$ and via a path of rotations in $V_j$. This unwinding can happen in each ring $R_j$ and the corresponding domain $V_j$ simultaneously for all $j$. The point is that on a given $R_j$, the unwinding will act via finitely many rotations (one for each ring $R_k$ such that $R_j \subset V_k$) and then via the unwinding on $R_j$.

\begin{figure}[h]
\begin{center}
\includegraphics[width=2in]{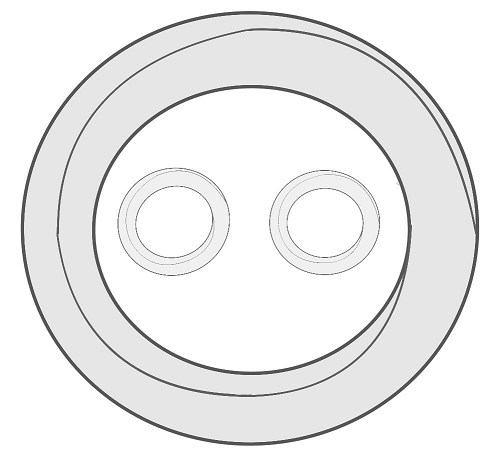}
\caption{Round rings and Dehn twists}
\label{fig:1}
\end{center}
\end{figure}

This idealized case is, however, not the most general case. Complications arise once $R_j$ are not round rings. In particular, it may certainly be the case that the two rings $R_{k_1},R_{k_2}$ contained in $V_j$, cannot be flowed isometrically around $D_j$, see Figure \ref{fig:2}

\begin{figure}[h]
\begin{center}
\includegraphics[width=3in]{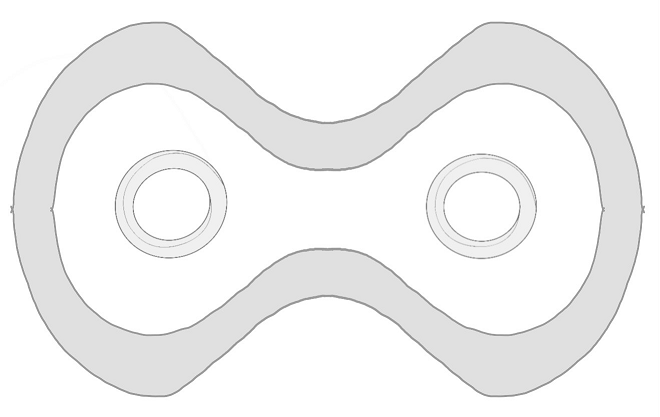}
\caption{Rings that are not round}
\label{fig:2}
\end{center}
\end{figure}

Our resolution to this issue is to use the hypothesis that $\dim_A X <1$ to show that the intersection $X \cap V_j$ may be covered by small islands which can be flowed into a relatively small ball contained in $V_j$. The point is that while the next level of rings down from $R_j$ may not be flowed around $V_j$, we can pass through finitely many levels, say $N$, to obtain a collection of rings which can be flowed around $V_j$.

Consequently, to unwind the Dehn twist in $R_j$, we concatenate three bi-Lipschitz paths in $V_j$: one to move the rings $N$ levels down into a given disk contained in $V_j$, one to act as a conjugate of rotations in $V_j$, and then the third to undo the first path. It follows that we may apply this construction simultaneously in the collection of levels that differ by $N$ to yield a bi-Lipschitz path. Applying this construction $N$ times, we may concatenate the resulting bi-Lipschitz paths to obtain one path from the identity to $f$ itself.

\subsection{Outline of the paper}

In Section \ref{sec:prelim}, we recall the basic definitions and properties of the objects we will use. In Section \ref{sec:UD-hyper}, we prove Theorem \ref{thm:UD-hyper}. In Section \ref{sec:BLpaths}, we prove some technical results on bi-Lipschitz paths. In Section \ref{sec:collapse}, we study how to collapse sets of Assouad dimension less than $1$ into small disks. Finally in Section \ref{sec:multitwist}, we prove Proposition \ref{prop:main}, and in Section \ref{sec:proof} we prove how the map $f$ in Theorem \ref{thm:main1} can be decomposed into bi-Lipschitz mappings of small isometric distortion. Finally, in Section \ref{sec:example} we construct a multitwist map with a singular set of Assouad dimension close to 2 that can be decomposed using the techniques of the paper.

\subsection{Acknowledgements}

The authors are indebted to Vladimir Markovic for suggesting this problem and many discussions thereupon.

\section{Preliminaries}\label{sec:prelim}

\subsection{Modulus of ring domains}

Given a family $\G$ of curves in $\R^n$, define the \emph{conformal modulus}
\[ \Mod(\G) = \inf_{\rho} \int_{\R^n}\rho(x)^n\, dx \]
where the infimum is taken over all Borel $\rho:\R^n \to [0,\infty)$ such that $\int_{\g}\rho\, ds \geq 1$ for all locally rectifiable $\g\in \G$.

Here and for the rest, given 
%a domain $\Omega \subset\R^n$ and two closed sets $F_1,F_2\subset \R^n$, we denote by $\D(F_1,F_2,\Omega)$ the family of curves in $\Omega$ that join $F_1$ with $F_2$. Given 
a ring domain $R$ in $\R^2$ with boundary components $\g_1$ and $\g_2$, we denote by $M(R)$ the modulus of  the family of curves in $R$ that join $\g_1$ with $\g_2$.
%\[ M(R) := \Mod(\D(\g_1,\g_2,R)).\]
Observe that the larger $M(R)$ is, the thinner the ring domain $R$ is.
It is well known \cite{Loewner} that there exists a decreasing function $\psi:(0,\infty) \to (0,\infty)$ such that, if $R$ is a ring domain with outer boundary component $\g_1$ and inner boundary component $\g_2$, then
\begin{equation}\label{eq:Loewner}
 M(R) \geq \psi\left( \frac{\dist(\g_1,\g_2)}{\diam{\g_2}} \right).
\end{equation}

\subsection{Assouad dimension}

A set $X\subset \R^N$ is \emph{$s$-homogeneous} for some $s\geq 0$ if there exists $C>0$ such that for every bounded set $A\subset X$, any $\e\in (0,\diam{A})$, and any $\e$-separated set $V\subset A$, 
\[\card{V} \leq C(\e^{-1}\diam{A})^{s}.\]
Recall that a set $V\subset A$ is $\e$-separated if for any distinct $x,y \in V$ we have $|x-y|\geq \e$.

If we want to emphasize on the constant $C$, we say that $X$ is \emph{$(C,s)$-homogeneous}. Note that every subset of $\R^N$ is $N$-homogeneous. Moreover, if $0\leq s_1\leq s_2$ and $X$ is $s_1$-homogeneous, then it is also $s_2$-homogeneous. %Recall from the Introduction that 
The \emph{Assouad dimension} of a set $X \subset \R^N$ is defined as
\[ \dim_A(X) = \inf\{s \geq 0 : X\text{ is $s$-homogeneous}\}.\]

\subsection{Hyperbolic geometry}

Suppose $X\subset\S^2$ is a Cantor set and $S=\S^2 \setminus X$ is a hyperbolic Riemann surface with a pants decomposition. Here, we recall how the cuffs $(\alpha_j)$ of the decomposition can be related to the thickness of ring domains embedded in the surface, see Figure \ref{fig:3}.

\begin{figure}[h]
\begin{center}
\includegraphics[width=5in]{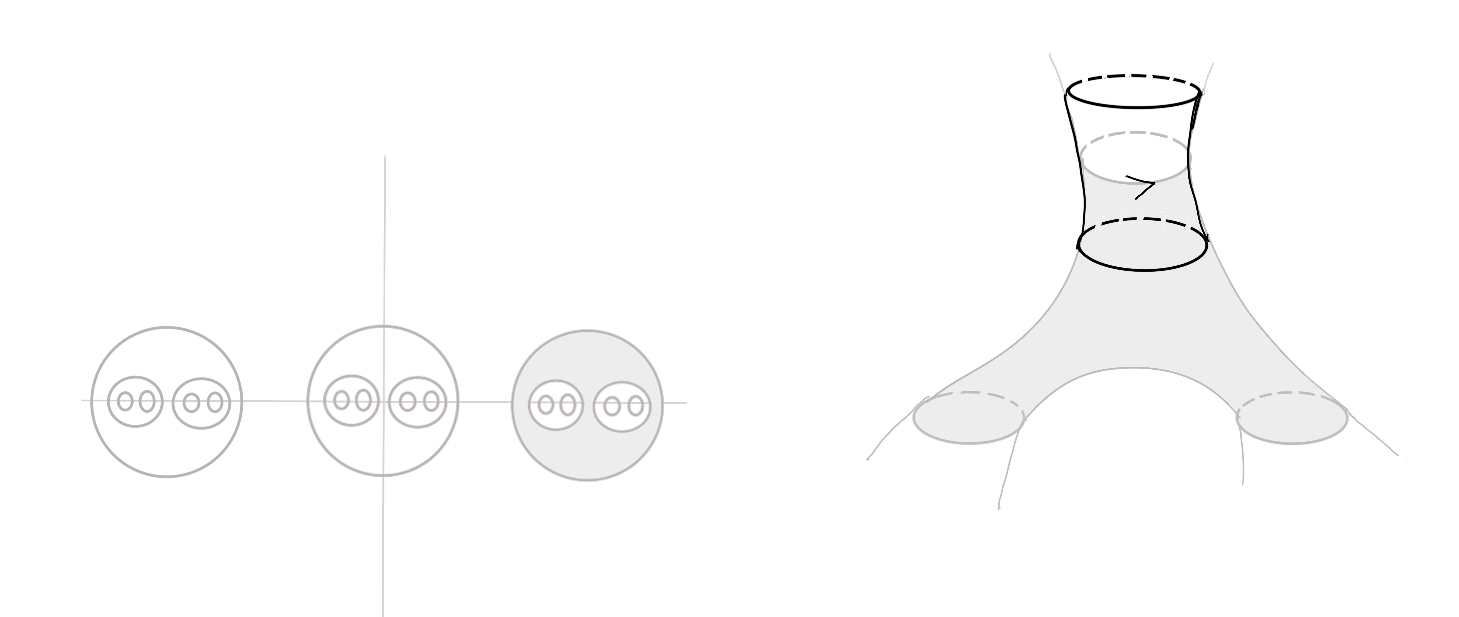}
\caption{On the left, we have a pants decomposition for $\S^2 \setminus X$ with a particular pair of pants shaded. On the right, we have a topological model of a pair of pants, again shaded, with the arrowed curve a geodesic cuff $\alpha_j$ and the ring domain in black an example of $R_j'$.}
\label{fig:3}
\end{center}
\end{figure}

\begin{proposition}\label{prop:collar}
For each $j$ there exists a ring domain $R_j' \subset \S^2\setminus X$ that contains $\alpha_j$ such that domains $R_j'$ are mutually disjoint and
\[ M(R_j') =  \frac{\ell(\alpha_j)}{2\arcsin ( e^{-\ell(\alpha_j)} )} .\]
\end{proposition}

This result is assuredly standard. Maskit \cite{Ma85} proves this for finite type surfaces, but since we will be applying this to infinite type surfaces, we give a proof for the convenience of the reader. We need the following Collar Lemma. 

\begin{lemma}[{\cite[Lemma 2.2]{ALPSS}}]\label{lem:collar}
There exist pairwise disjoint collars $(C_j)_j$ of cuffs $(\alpha_j)_j$
%If $\a_j$ is a geodesic boundary component of one of the pairs of pants, then there is a collar $C$ of $\a_j$ 
given by
\[ C_j = \{ z\in S : d_S(z, \gamma)  \leq B(\ell(\alpha_j)) \}, \]
where $d_S$ denotes the hyperbolic metric on $S$ and 
\[ B(t) = \frac{1}{2} \log \left ( 1 + \frac{2}{e^t-1} \right ).\]
\end{lemma}
%Moreover, \cite[Lemma 2.2]{ALPSS} asserts that the collection of collars formed in this manner is pairwise disjoint.

\begin{proof}[Proof of Proposition \ref{prop:collar}]
Let $C_j$ be the collars from Lemma \ref{lem:collar}. These collars are necessarily ring domains.

Since $S$ is a hyperbolic Riemann surface, we can consider its lift to the strip model of the hyperbolic plane. More precisely, let $\Sigma = \{ z\in \C : | \Im(z)| <\pi /2 \}$. Then the hyperbolic metric density on $\Sigma$ is given by $\lambda_{\Sigma}(z) = \sec (\Im (z) )$ (see for example \cite[Example 7.9]{BM}). Since we can identify $S$ with $\Sigma / G$, where $G$ is a covering group of deck transformations, we can lift $\alpha_j$ so that its lift is contained in the real axis in $\Sigma$. Moreover, $C_j$ can be lifted to a rectangle in $\Sigma$ whose closure is given by $R = [-r,r] \times [-s, s]$.

Here, we have $d_{\Sigma} (-r , r) = \ell$ and $d_{\Sigma}(-is , is ) = 2 B(\ell(\alpha_j) )$. Since the hyperbolic metric and the Euclidean metric coincide on the real axis in $\Sigma$, we have $r=\ell(\alpha_j) / 2$. Next,
\[2 B(\ell(\alpha_j)) = d_{\Sigma}(-is, is) = \int_{-s}^s \sec t \: dt = 2 \ln (\sec s +\tan s).\]
Solving this for $s$, we see that
\[ s = \arcsin ( \tanh (B(\ell(\alpha_j))))\]
and hence
\[ s = \arcsin ( e^{-\ell(\alpha_j)} ).\]
Finally, $M(C_j)$ is equal to the modulus of the path family joining the $r$-sides to the $s$-sides of the rectangle $R$. Thus
\[ M(C_j) = \frac{r}{s} = \frac{\ell_S(\alpha_j)}{2\arcsin ( e^{-\ell_S(\alpha_j)} )} . \qedhere\]
\end{proof}

We will also need the following result of Wolpert.

\begin{lemma}[{Wolpert \cite{Wolpert}}]
\label{thm:wolpert}
Let $f:S \to S'$ be a $K$-quasiconformal homeomorphism between hyperbolic Riemann surfaces $S,S'$. Let $\alpha$ be a closed geodesic in $S$, and let $\alpha'$ be the unique closed geodesic in $S'$ that is homotopic to $f(\alpha)$. Then
\[ K^{-1}\ell_S[ \alpha ] \leq \ell_{S'} [ \alpha' ] \leq K \ell_S [ \alpha ].\]
\end{lemma}

\subsection{Square thickenings}

We recall some terminology and notation from \cite{MM2}. Given $\alpha>0$ define 
\begin{align*} 
\mathscr{G}_{\alpha} := \{ \alpha\textbf{n} + [0,\alpha]^2 : \textbf{n} \in \mathbb{Z}^2 \}\quad\text{and}\quad  \mathscr{G}_{\alpha}^1 := \{e : \text{$e$ is an edge of some $S\in \mathscr{G}_{\alpha}$} \}.
\end{align*}
Given a set $W \subset \R^2$ define $W^{\alpha}$ to be the collection of all squares in $\mathscr{G}_{\alpha}$ that intersect with $W$. For $\delta>0$, define the $\delta$-square thickening
\[ \mathcal{T}_{\d}(W) = (W^{4\d})^{\d},\] 
see Figure \ref{fig:4}.

%\begin{figure}[h]
%\begin{center}
%\includegraphics[width=3in]{fv_bilip_3.png}
%\caption{The shaded region is $W$. The thicker grid consists of edges in $\mathscr{G}_{4\d}^1$ and the thinner grid consists of edges in $\mathscr{G}_{\d}^1$. The grey curve is the boundary of $W^{4\d}$ and the black curve is the boundary of $\mathcal{T}_{\d}(W)$.}
%\label{fig:4}
%\end{center}
%\end{figure}

\begin{figure}[h]
\begin{center}
\includegraphics[width=3in]{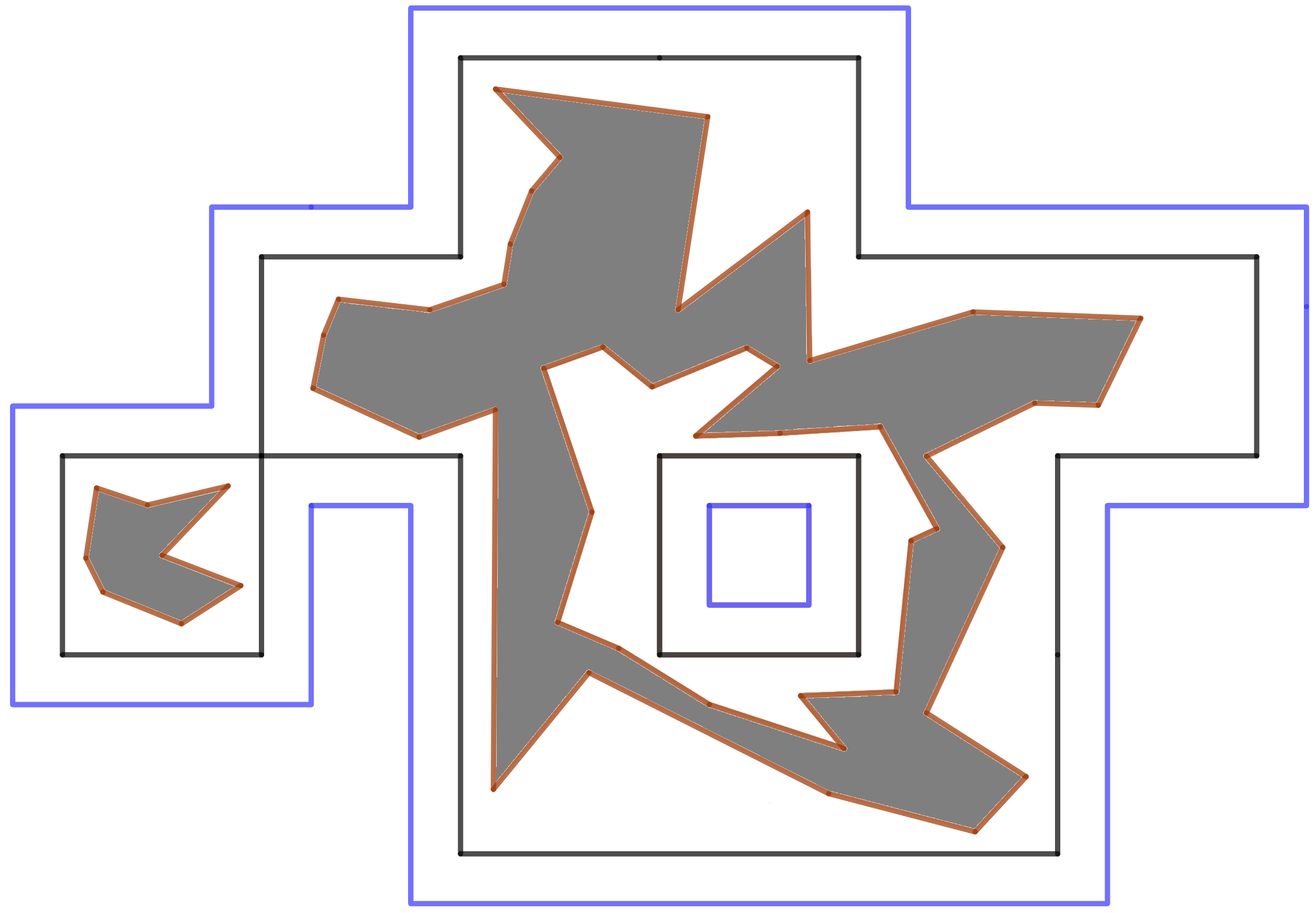}
\caption{The shaded region is $W$. 
%The thicker grid consists of edges in $\mathscr{G}_{4\d}^1$ and the thinner grid consists of edges in $\mathscr{G}_{\d}^1$. 
The black curve is the boundary of $W^{4\d}$ and the blue curve is the boundary of $\mathcal{T}_{\d}(W)$.}
\label{fig:4}
\end{center}
\end{figure}

\begin{lemma}[{\cite[Lemma 2.1]{MM2}}]\label{lem:MM}
If $W$ is a bounded subset of the plane and $\d>0$, then the boundary of $\mathcal{T}_{\d}(W)$ is a finite union of mutually disjoint polygonal Jordan curves made of edges in $\mathscr{G}^1_{\d}$ and
\begin{equation}\label{eq:MM} 
\d \leq \dist(x,W) \leq 8\d\qquad \text{for all $x\in\partial\mathcal{T}_{\d}$}.
\end{equation}
\end{lemma}

%\begin{proof}
%The first claim is clear. Let us verify (\ref{eq:MM}). Let $x\in\R^2$ be a point with $\dist(x,W)<\d$. Then, there exists a square $\Sigma \in \mathscr{G}_{4\d}$ such that $\dist(x,\Sigma)<\d$. But then, $x$ is in the interior of $\Sigma^{\d}$, hence in the interior of $\mathcal{T}_{\d}(W)$. Therefore, for any $x\in \partial \mathcal{T}_{\d}(W)$, $\dist(x,W)\geq \d$. For the upper bound of (\ref{eq:MM}), we first note that for any $\Sigma \in  \mathscr{G}_{4\d}$ with $\Sigma \subset W^{4\d}$, and for any $x\in \Sigma$, $\dist(x,W) \leq 4\sqrt{2}\d$. Therefore, for any $\Sigma \in \mathcal{T}_{\d}(W)$ and any $x\in \Sigma$,
%\[ \dist(x,W) \leq \dist(x, W^{4\d}) + \sup_{y\in W^{4\d}}\dist(y,W) \leq \sqrt{2}\d + 4\sqrt{2}\d = 5\sqrt{2}\d.\qedhere\]
%\end{proof}

%\begin{lem}\label{lem:ballthick}
%Let $\G$ be a polygonal Jordan curve in $\R^2$ with edges in $\mathcal{G}^1_{\d}$. Then, the neighborhood $N(\G,\frac14\d) = \bigcup_{x\in \G} B(x, \frac14\d)$ is ring domain such that for any 
%\end{lem}

\subsection{Symbolic notation}
\label{sect:sn}

At several junctures in this paper, it will be convenient to use symbolic notation to describe our constructions.

Given an integer $k\geq 0$, we denote by $\{1,2\}^k$ the set of words formed from the alphabet $\{1,2\}$ that have length exactly $k$. Conventionally, we set $\{1,2\}^0 = \{\varepsilon\}$ where $\varepsilon$ is the empty word. We also denote by $\{1,2\}^* = \bigcup_{k\geq 0}\{1,2\}^k$ the set of all finite words formed from $\{1,2\}$. Given a word $w \in \{1,2\}^*$, we denote by $|w|$ the length of $w$ with the convention $|\varepsilon| = 0$.

\section{Uniformly disconnected Cantor sets and hyperbolic geometry}\label{sec:UD-hyper}

In this section, we prove Theorem \ref{thm:UD-hyper}. One direction of the theorem is given in Section \ref{sec:UDimpliesH} and the other direction is given in Section \ref{sec:HimpliesUD}.

\subsection{Assuming uniformly disconnected}\label{sec:UDimpliesH}

Here we prove the necessary direction of Theorem \ref{thm:UD-hyper}.

\begin{proposition}
\label{thm:1}
Let $X\subset \mathbb{S}^2$ be a $c$-uniformly disconnected Cantor set. There exists $M>0$ depending only on $c$, and there exists a pants decomposition for the Riemann surface $S = \S^2 \setminus X$ such that the associated cuffs $(\alpha_j)$ satisfy $\sup_j \ell [ \alpha_j ] < M$.
\end{proposition}

Denote by $\mathcal{C}$ the standard one-third Cantor set and by $S_0$ the Riemann surface $S_0 = \S^2 \setminus  \mathcal{C}$.

\begin{lemma}[{\cite[Corollary A]{Vext}}]\label{lem:ud}
If $X\subset \S^2$ is a uniformly disconnected set, then there exists a quasiconformal map $f:\S^2 \to\S^2$ such that $f(X)\subset \mathcal{C}$.
\end{lemma}

As observed in \cite[p.5]{Shiga}, the pairs of pants in the pants decomposition of $S_0$ can be chosen to be conformally equivalent to one another. It follows that each such pair of pants has the same cuff lengths. To see this, suppose $P$ and $P'$ are two pairs of pants in this decomposition with a conformal map $h: P \to P'$. Let $R$ and $R'$ be the respective doubles of $P$ and $P'$, that is, $R$ and $R'$ are genus two surfaces. Then $h$ extends via reflection to a conformal map $\widetilde{h}: R \to R'$ and hence $\widetilde{h}$ is a hyperbolic isometry. Restricting $\widetilde{h}$ to the cuffs of $P$, we see that $P$ and $P'$ have the same cuff lengths.

In particular, we conclude that there exist a constant $q>0$ and a pair of pants decomposition of $S_0$ with cuffs $(C_j)$ such that
\begin{equation}
\label{eq:cj}
\sup_j \ell_{S_0} [ C_j ] = q.
\end{equation} 

\begin{proof}[Proof of Proposition \ref{thm:1}]
Let $X\subset \R^2$ be a uniformly disconnected Cantor set, and let $f$ be the quasiconformal map from Lemma \ref{lem:ud} such that $f(X) \subset \mathcal{C}$. We will use the pants decomposition with cuffs $(C_j)$ for $S_0$. Since $f(S) \supset S_0$, we may use a subset of the $(C_j)$ to generate a pair of pants decomposition for $f(S)$. This subset can be labelled as $(C_{j_k})$ and we, for brevity, will denote it by $(\beta_k)$.

Suppose $(\eta_j)$ are the cuffs of a pants decomposition of $S$. Then each $\eta_j$ is homotopic to $f^{-1}(\beta_k)$ for some $k$ and vice versa. Hence, if we assume for a contradiction that $\ell_S [ \eta_{j_m} ] \to \infty$, it follows via Lemma \ref{thm:wolpert} that $\ell_{f(S)} [ \beta_{k_m} ] \to \infty$. 

Since $f(S) \supset S_0$, the subordination principle for the hyperbolic metric implies that if $\gamma$ is any path in $S_0$, then $\ell_{f(S)} ( \gamma ) \leq \ell_{S_0} (\gamma)$. In particular, we conclude that $\ell _{S_0} [ \beta_{k_m} ] \to \infty$. This contradicts \eqref{eq:cj}.
\end{proof}

\subsection{Towards uniformly disconnected}\label{sec:HimpliesUD} 

Here we prove the sufficient direction of Theorem \ref{thm:UD-hyper}.

\begin{proposition}\label{prop:UDsuff}
Let $X \subset \R^2$ be a Cantor set and suppose that the Riemann surface $S=\S^2\setminus X$ has a pants decomposition $(P_j)$ where the cuffs $(\alpha_j)$ satisfy $\sup_j\ell [ \alpha_j ] <L<\infty$. Then $X$ is $c$-uniformly disconnected for some $c$ depending only on $L$.
\end{proposition}

%Now assume we have a pants decomposition of $S=\S^2\setminus X$ with a uniform upper bound $T>0$ on $\ell_S (\alpha_j)$. By Proposition, \ref{prop:collar}, we have an upper bound $T'$ on the moduli of the ring domains $R_j$.

%Given a set $\mathcal{W}\subset \{1,2\}^*$ and an integer $k\geq 0$, we denote $\mathcal{W}_k=\mathcal{W}\cap \{1,2\}^k$.
%
%A \emph{dyadic decomposition} of a metric space $X$ is a collection $\{X_w : w\in\mathcal{W}\}$ of nonempty subsets of $X$ that satisfy the following properties.
%\begin{enumerate}
%\item The empty word $\varepsilon \in \mathcal{W}$ and $X_{\varepsilon} = X$.  
%\item If $wi \in \mathcal{W}$ for some $i\in \{1,2\}$ and $w\in\{1,2\}^*$, then $w\in\mathcal{W}$ and $X_{wi}\subset X_w$.
%\item Suppose that $w\in\mathcal{W}$. 
%\begin{enumerate}
%\item If $X_w$ is a point, then $w1\in\mathcal{W}$, $w2 \not\in\mathcal{W}$, and $X_{w1}=X_w$. 
%\item If $X_w$ has at least two points, then $w1,w2 \in \mathcal{W}$ and $X_w$ it is the disjoint union of $X_{w1}$ and $X_{w2}$.
%\end{enumerate}
%\end{enumerate}

Recall the symbolic notation from Section \ref{sect:sn}.

\begin{lemma}\label{lem:UDreldist}
A totally bounded metric space $X$ is uniformly disconnected if and only if there exists a set $\mathcal{W}\subset \{1,2\}^*$, a constant $\d>0$ and a collection of subsets $\{X_w :w\in\mathcal{W}\}$ with the following properties. 
\begin{enumerate}
\item The empty word $\varepsilon \in \mathcal{W}$ and $X_{\varepsilon} = X$.  
\item If $wi \in \mathcal{W}$ for some $i\in \{1,2\}$ and $w\in\{1,2\}^*$, then $w\in\mathcal{W}$ and $X_{wi}\subset X_w$.
\item If $X_w$ is a point for some $w\in\mathcal{W}$, then $w1\in\mathcal{W}$, $w2 \not\in\mathcal{W}$, and $X_{w1}=X_w$. 
\item If $X_w$ has at least two points for some $w\in\mathcal{W}$, then $w1,w2 \in \mathcal{W}$, $X_w = X_{w1} \cup X_{w2}$, and
\begin{equation}\label{eq:maxreldist}
\dist(X_{w1},X_{w2}) \geq\d \max\{\diam{X_{w1}}, \diam{X_{w2}}\}.
\end{equation}
\end{enumerate}
The constant of uniform disconnectedness and $\d$ are quantitatively related.
\end{lemma}

\begin{proof}
Assume first that $X$ is $c$-uniformly disconnected. 
%Note that every subset of $X$ is also $c$-uniformly disconnected. 
Set $X_{\varepsilon} = X$. Assume now that for some $w\in\{1,2\}^*$, we have defined a nonempty set $X_w \subset X$. If $X_w$ is a single point, then set $X_{w1}=X_w$. Assume now that $X_w$ contains at least two points and fix $x\in X_w$. By the uniform disconnectedness of $X_w$, there exists $E \subset X_w$ such that $x\in E$, $\diam{E}\leq \frac12\diam{X_w}$ and $\dist(E,X_w\setminus E)\geq (2c)^{-1}\diam{X_w}$. Set $X_{w1} = E$ and $X_{w2}=X_w\setminus E$. Note that
\[ \dist(X_{w1},X_{w2}) \geq (2c)^{-1}\diam{X_w} \geq (2c)^{-1}\max\{\diam{X_{w1}}, \diam{X_{w2}}\}.\]
Setting $\mathcal{W}$ to be the set of all words $w \in \{1,2\}^*$ for which $X_w$ has been defined, it is easy to see that $\{X_w : w\in\mathcal{W}\}$ satisfies (i)--(iv) with $\d = (2c)^{-1}$.

Suppose now that there exists $\mathcal{W}\subset\{1,2\}^*$, $\d>0$ and a collection $\{X_w: w\in\mathcal{W}\}$ satisfying (i)--(iv). We first show that if $(i_n)$ is a sequence in $\{1,2\}$ such that $i_1\cdots i_n \in \mathcal{W}$ for all $n\in\N$, then $\lim_{n\to \infty}\diam{X_{i_1\cdots i_n}} =0$. Assume for a contradiction that there exists $d>0$ and a sequence $(i_n)$ in $\{1,2\}$ such that $i_1\cdots i_n \in \mathcal{W}$ and $\diam{X_{i_1\cdots i_n}} >d $ for all $n\in\N$. Fix $x_1 \in X_{\varepsilon} \setminus X_{i_1}$ and for each $n\in\N$ fix $x_n \in X_{i_1\cdots i_{n-1}} \setminus X_{i_1\cdots i_n}$. By (\ref{eq:maxreldist}), for any distinct $i,j \in \N$, $|x_i - x_j| \geq\d d$. Then, the set $\{x_n : n\in\N\}$ is not totally bounded and we reach a contradiction.

%Since $X$ is doubling, by Assouad's Embedding Theorem \cite[Theorem 12.1]{Heinonen}, there exists $n\in\N$ and a quasisymmetric embedding $f: X \to \R^n$. Let $Y = f(X)$ and for each $w\in\mathcal{W}$, set $Y_w = f(X_w)$. We first show that the decomposition $(Y_w)_{w\in\mathcal{W}}$ satisfies (\ref{eq:maxreldist}) for some $\d'$ depending on the quasisymmetric control $\eta$ and on $\d$. Towards this end, fix $w\in \mathcal{W}$, $i\in \{1,2\}$, $x_1\in X_{w1}$ and $x_2 \in X_{w2}$.

We prove now that $X$ is uniformly disconnected. If $X$ contains a single point, then the claim is trivial. Assume now that $\diam{X}>0$ and let $x \in X$ and $r \in (0,\diam{X})$. Let $w \in \mathcal{W}$ be the maximal word (in word-length) such that $x\in X_w$ and $\diam{X_w}\geq r$.
%, and for each $i\in\{1,2\}$, $x \not\in X_{wi}$.
%$w=i_1\cdots i_n \in \mathcal{W}$ be a word of maximal length such that $x \in X_w$ and $\diam{X_w} \geq r$. 
Write $w=i_1\cdots i_n$ and assume that $x\in X_{wi_{n+1}}$ where $i_{n+1}\in\{1,2\}$. Setting $E = X_{wi_{n+1}}$, we have that $\diam{E} < r$ while
\begin{align*} 
\dist(E,X \setminus E) &= \min_{j=1,\dots,n+1} \dist(X_{i_1\cdots i_j}, X_{i_1\cdots i_{j-1}} \setminus X_{i_1\cdots i_j})\\
&\geq \min\left\{ \dist(X_{w1},X_{w2}), \d\min_{j=1,\dots,n} \diam{X_{i_1\cdots i_j}}\right\}\\
&= \min\left\{ \dist(X_{w1},X_{w2}), \d\diam{X_{w}}\right\}.
\end{align*}
%\min_{j=1,\dots,n}\dist(E, X\setminus X_{i_1\cdots i_j})\\
%&\geq \d\max\left\{\dist(X_{w1},X_{w2}), \max_{j=1,\dots,n}\diam{X_{i_1\cdots i_j}}\right\}\\
%&\geq \d\max\{\diam{X_w},\dist(X_{w1},X_{w2})\}.
%\end{align*}
By the triangle inequality and (\ref{eq:maxreldist}),
\[ \diam{X_w} \leq \diam{X_{w1}} + \dist(X_{w1},X_{w2}) + \diam{X_{w2}} \leq (1+2\d^{-1})\dist(X_{w1},X_{w2}).\]
Therefore,
\[ \dist(E,X\setminus E) \geq \min\left\{\d,(1+2\d^{-1})^{-1}\right\}\diam{X_w} \geq (1+2\d^{-1})^{-1}r\]
and $X$ is $c$-uniformly disconnected with $c=(1+2\d^{-1})^{-1}$.
\end{proof}

\begin{comment}

\begin{remark}
Lemma \ref{lem:UDreldist} fails if ``max'' is replaced by ``min'' in (\ref{eq:maxreldist}).  Indeed, set 
\[ X = [-1,0]\cup\{1/n : n\in\N\}\]
and define a collection $\{X_w : w\in\mathcal{W}\}$ by the rules:
\begin{enumerate}
\item $X_{\varepsilon} = X$;
\item if we have defined $X_w$ for some $w\in\{1,2\}^*$ and if $w$ does not contain the digit ``1'', then $X_{w1}=\{\frac{1}{|w|+1}\}$ and $X_{w2}=X_w\setminus X_{w1}$; 
\item if we have defined $X_w$ for some $w\in\{1,2\}^*$ and if $w$ contains the digit ``1'', then $X_{w1}=X_w$ and $w2\not\in\mathcal{W}$.
\end{enumerate}
It is easy to see that $\{X_w : w\in\mathcal{W}\}$ satisfies (i)--(iv) of Lemma \ref{lem:UDreldist} except that (\ref{eq:maxreldist}) holds with ``min''. However, $X$ is not uniformly disconnected.
\end{remark}

\end{comment}

We now show Proposition \ref{prop:UDsuff} and thus complete the proof of Theorem \ref{thm:UD-hyper}.

%\begin{proposition}\label{prop:mod-ud}
%Let $(D_{w})_{w\in\mathcal{W}}$ be a family of Jordan domains in $\R^2$ such that for any $w\in\mathcal{W}$ and any distinct $i,j \in A$ we have $\overline{D_{wi}}\subset D_{w}$ and $\overline{D_{wi}} \cap \overline{D_{wj}} = \emptyset$. Assume that
%\[ \inf_{w \in \mathcal{W}}M(\Gamma_{w}) > 0,\]
%where $\Gamma_{w}$ is the set of all curves in $D_{w}\setminus (D_{w1}\cup D_{w2})$ that are homotopic to $\partial D_{w}$. Then $\bigcap_{k=0}^{\infty}\bigcup_{w\in\mathcal{W}_k}D_{w}$ is uniformly disconnected.
%\end{proposition}

\begin{proof}[{Proof of Proposition \ref{prop:UDsuff}}]
We assume that each $\alpha_j$ has been chosen to minimize $\ell_S (\alpha_j)$ in its homotopy class. By Proposition \ref{prop:collar} there exist disjoint ring domains $R_j'$ in $\R^2\setminus X$ containing the cuffs $\alpha_j$ such that $\sup_j M(R_j') \leq m$ for some $m$ depending only on $L$. For each $j$ let $V_j'$ and $U_j'$ be the bounded and unbounded, respectively, components of $\R^2\setminus R_j'$. We relabel the ring domains $R_j'$ in the following way.

Firstly, we remark that there exist three indices $j_1,j_2,j_3$ such that $V_{j_1}',V_{j_2}',V_{j_3}'$ are mutually disjoint and $X$ is contained in $V_{j_1}'\cup V_{j_2}'\cup V_{j_3}'$. For $l=1,2,3$, we denote $R_{l,\varepsilon}':=R_{j_l}'$. Inductively, assume that for some $l\in \{1,2,3\}$ and some $w\in\{1,2\}^*$ we have defined $R_{l,w}' = R_{j}'$ for an index $j$. There exist two indices $i_1,i_2$ such that 
\begin{enumerate}
\item $V_{i_1}',V_{i_2}' \subset V_j'$ and $V_{i_1}'\cap V_{i_2}' = \emptyset$; 
\item if $V_i' \subset V_j'$ for some $i$, then $V_i'\subset V_{i_1}'$ or $V_i' \subset V_{i_2}'$.
\end{enumerate}
Set now $R_{l,w1}':=R_{i_1}'$ and $R_{l,w2}':=R_{i_2}'$. 

If for some $j$, $l\in\{1,2,3\}$ and $w\in \{1,2\}^*$ we have defined $R_{l,w}' = R_j'$, then define $V_{l,w}':=V_j'$ and $U_{l,w}':=U_j'$. Set also $X_{l,w} = X\cap V_{l,w}'$. It is easy to see that for each $l=1,2,3$, the collection $\{X_{l,w}: w\in\{1,2\}^*\}$ satisfies (i)--(iii) of Lemma \ref{lem:UDreldist} for $X_{l,\varepsilon}$.

Fix now $l\in\{1,2,3\}$.
%For the converse, suppose now that there exists $m>0$  and a dyadic decomposition $(X_{w})_{w\in\mathcal{W}}$ of $X$ such that for every $w\in\mathcal{W}$ there exists an annulus $R_w \subset \R^2$ separating $X_w$ from $X\setminus X_w$ with $M(R_w) \geq m$. Let $U_w$ and $D_w$ be the unbounded and bounded, respectively, components of $\R^2 \setminus R_w$. 
By (\ref{eq:Loewner}) we have that there exists $d>0$ depending only on $m$ such that for all $w\in\{1,2\}^*$
\begin{equation*} 
\dist(U_{l,w}',V_{l,w}') \geq d\diam{V_{l,w}'}.
\end{equation*}
Therefore, for each $w\in\{1,2\}^*$
\begin{align*} 
\dist(X_{l,w1},X_{l,w2})  \geq \max_{i=1,2} \dist(X_{l,wi},\partial V_{l,wi}') &\geq  \max_{i=1,2}\dist(V_{l,wi}', U_{l,wi}')\\
&\geq d\max_{i=1,2}\diam{V_{l,wi}'}  \\
&\geq d\max_{i=1,2} \diam{X_{l,wi}}.
\end{align*}

Working as above, we can deduce that for all distinct $l,l'\in\{1,2,3\}$, we have
\[ \dist(X_{l,\varepsilon},X_{l',\varepsilon}) \geq d \max\{ \diam{X_{l,\varepsilon}}, \diam{X_{l',\varepsilon}}\}.\]
Since $X$ is compact, by Lemma \ref{lem:UDreldist}, $X$ is $C$-uniformly disconnected with $C$ depending only on $d$, hence only on $m$, hence only on $L$. 
\end{proof}

%The proof of Theorem \ref{thm:UD-hyper} is complete.

\section{Bi-Lipschitz paths}\label{sec:BLpaths}

Our strategy to proving Theorem \ref{thm:main1} is to use bi-Lipschitz paths to yield the required decomposition.
We recall the following definition from \cite{FM}.

\begin{definition}
\label{bilippath}
Let $(X,d_{X})$ be a metric space.
A path $H:[0,1] \rightarrow LIP(X)$
is called a \emph{bi-Lipschitz path} if for every $\epsilon >0$, there exists $\delta >0$ such that if $s,t \in [0,1]$ with $\av s-t \av < \delta$, the following two conditions hold:
\begin{enumerate}
\item for all $x \in X$, $d_{X}(H_{s} \circ H_{t}^{-1}(x),x) <\epsilon$;
\item we have that $H_{s} \circ H_{t}^{-1}$ is $(1+\epsilon)$-bi-Lipschitz with respect to $d_{X}$.
\end{enumerate}
\end{definition}

In this paper bi-Lipschitz paths are denoted by capital letters $F,G,H, \cdots$. Given two bi-Lipschitz maps $f,g:X\to X$, two bi-Lipschitz paths $F,G: [0,1] \to LIP(X)$, and a subset $E\subset X$, we define
\begin{enumerate}
\item the concatenation of $F$ with $G$ to be the biLipschitz path $H: [0,1] \to LIP(X)$ with $H_t = F_{2t}$ for $t\in [0,1/2]$ and $H_{t} = G_{2t-1}$ for $t\in [1/2,1]$, and we may then concatenate finitely many bi-Lipschitz paths in the obvious way;
\item the restriction $F|E : [0,1] \to LIP(E)$ by $(F|E)_t = F_t|E$;
\item the composition $F\circ G$ by $(F\circ G)_t = F_t \circ G_t$ for all $t\in [0,1]$;
\item the composition $f\circ F\circ g$ by $(f\circ F\circ g)_t = f\circ F_t \circ g$ for all $t\in [0,1]$.
\end{enumerate}

We emphasize that in (iii) and (iv) here, the compositions need not be bi-Lipschitz paths. Much of our work will involve showing that our constructions are made carefully enough that when we do need to compose or conjugate, we do still have a bi-Lipschitz path. For illustrative purposes, we include examples where (iii) and (iv) fail to give a bi-Lipschitz path.

\begin{example}
\label{ex:1}
Let $L>1$, $B = B(0,1/3) \subset \R^2$ and let $f:\R^2\to \R^2$ be an $L$-bi-Lipschitz map which is the identity on $\R^2 \setminus B$. Then define
\begin{equation*} 
\widetilde{f}(z) = \begin{cases}
 f( z-n) + n, & \text{$z\in B(n,1/3)$, $n\in \{0,1,2,\ldots \}$} \\ 
z,  & \text{otherwise.} \end{cases} 
\end{equation*}
Clearly $\widetilde{f}$ is also an $L$-bi-Lipschitz map.

Set $F_t =  e^{i\pi t} z$ for $t\in [0,1]$ and $H = \text{Id}\circ F\circ \tilde{f}$ where $\text{Id}:\R^2 \to \R^2$ is the identity map. Then,
\[ H_t^{-1} \circ H_s (z) = \widetilde{f}^{-1} ( e^{i(s-t)\pi} \widetilde{f}(z)).\]
Suppose $\epsilon <L-1$, $\delta >0$ and $|s-t| = \delta$. Then there exists $N\in \N$ large enough that $ B(Ne^{i(s-t)\pi} , 1/3) \cap B(N,1/3) = \emptyset$. Hence on $B(N,1/3)$ we have that $H_t^{-1} \circ H_s$ agrees with a composition of a rotation and $\widetilde{f}$. This means that $H_t^{-1} \circ H_s$ is not $(1+\epsilon)$-bi-Lipschitz and hence $H$ is not a bi-Lipschitz path.
We conclude that $\text{Id}\circ F\circ \tilde{f}$ is not a bi-Lipschitz path.

Using the same example and setting $G:[0,1] \to LIP(\R^2)$ being the constant path $\tilde{f}$ we see that $F\circ G$ is not a bi-Lipschitz path. Hence, compositions of bi-Lipschitz paths are not always bi-Lipschitz paths.
\end{example}

It is worth pointing out that in a bi-Lipschitz path, the elements are bi-Lipschitz with uniform constant.

\begin{lemma}
\label{lem:pathbound}
Suppose $H:[0,1] \to LIP(X)$ is a bi-Lipschitz path. Then there exists $L>1$ such that $H_t$ is an $L$-bi-Lipschitz map for each $t\in [0,1]$.
\end{lemma}

\begin{proof}
Clearly $H_0$ is an $L_0$-bi-Lipschitz map for some $L_0>1$. Set $\epsilon =1$ and the corresponding $\delta>0$ so that condition (ii) holds. In particular, for every $t_1\in(0,\delta)$, by condition (ii) applied to $(H_{t_1} \circ H_0^{-1}) \circ H_0$, the map $H_t$ is $2L_0$-bi-Lipschitz. Next, for every $t_2\in [\delta, 2\delta)$, there exists $t_1 \in(0,\delta)$ with $|t_2 - t_1| < \delta$. Applying condition (ii) to $(H_{t_2} \circ H_{t_1}^{-1}) \circ  ( H_{t_1} \circ H_0^{-1}) \circ H_0$, we see that $H_{t_2}$ is $2^2L_0$-bi-Lipschitz.

Continuing inductively, we see that for any $t\in [0,1]$, $H_t$ is $2^{\lfloor1/\delta \rfloor +1} L_0$-bi-Lipschitz.
\end{proof}

Next we show that if the restrictions of $H$ on three sets whose union is $\R^2$ are bi-Lipschitz paths, then $H$ is a bi-Lipschitz path quantitatively.

\begin{lemma}\label{lem:partition}
Let $A,B,C \subset \R^2$ be closed sets such that $\R^2 = A\cup B\cup C$. Suppose that for any $\e>0$, there exists $\d>0$ such that if $s,t\in [0,1]$ with $|s-t| <\delta$, then the two conditions in Definition \ref{bilippath} hold simultaneously for $H|A$, $H|B$, $H|C$. Then
% for any $\e>0$, there exists $\d>0$ such that if $s,t\in [0,1]$ with $|s-t| <\delta$ then
the two conditions in Definition \ref{bilippath} also hold for $H$ with the same $\d$.
\end{lemma}

\begin{proof}
Fix $\e>0$ and let $\d>0$ such that  the two conditions in Definition \ref{bilippath} hold simultaneously for $H|A$, $H|B$, $H|C$. Let $s,t \in [0,1]$ such that $|s-t|<\d$. If $x\in \R^2$, then without loss of generality we may assume that $x\in A$ and we have
\[ |H_s \circ H_t^{-1}(x) - x| = |(H|A)_s \circ (H|A)_t^{-1}(x) - x| < \e\]
and $H$ satisfies (i). 

For (ii), it suffices to show that $H_s\circ H_t^{-1}$ is $(1+\e)$-Lipschitz. Let $x,y \in \R^2$. If both $x$ and $y$ belong to the same set from $A,B,C$, then (ii) follows immediately. Assume without loss of generality that $x\in A$ and $y\in B\setminus A$. Let $z \in [x,y] \cap A$ such that $A\cap ([y,z] \setminus \{z\}) = \emptyset$. Since $y\in B\setminus A$, we have that $z\neq y$. 

There are now two cases. First, if $z\in B$ then
\begin{align*}
|H_{s} \circ H_{t}^{-1}(x) -H_{s} \circ H_{t}^{-1}(y) |&\leq   |H_{s} \circ H_{t}^{-1}(x) -H_{s} \circ H_{t}^{-1}(z) |\\
&\quad +  |H_{s} \circ H_{t}^{-1}(z) -H_{s} \circ H_{t}^{-1}(y) |\\  
&\leq   |(H|A)_{s} \circ (H|A)_{t}^{-1}(x) -(H|A)_{s} \circ (H|A)_{t}^{-1}(z) |\\
&\quad +  |(H|B)_{s} \circ (H|B)_{t}^{-1}(z) -(H|B)_{s} \circ (H|B)_{t}^{-1}(y) |\\
&\leq (1+\e)(|x-z| + |z-y|)\\
&= (1+\e)|x-y|.
\end{align*}
Second, if $z\in C$ then we have two sub-cases. If also $y\in C$, we have the same argument as above, with the role of $B$ played by $C$. If $y\notin C$, then let $w\in C\cap [z,y] $ be such that $C\cap ( [w,y] \setminus  \{ w \} ) = \emptyset$. We have $w\neq y$ by construction. Since $z\in A\cap C$ and $w \in B\cap C$, we have
\begin{align*} 
|H_{s} \circ H_{t}^{-1}(x) -H_{s} \circ H_{t}^{-1}(y) | &\leq   |H_{s} \circ H_{t}^{-1}(x) -H_{s} \circ H_{t}^{-1}(z) |\\
&\quad +  |H_{s} \circ H_{t}^{-1}(z) -H_{s} \circ H_{t}^{-1}(w) |\\  
&\quad +  |H_{s} \circ H_{t}^{-1}(w) -H_{s} \circ H_{t}^{-1}(y) |\\
 &\leq   |(H|A)_{s} \circ (H|A)_{t}^{-1}(x) -(H|A)_{s} \circ (H|A)_{t}^{-1}(z) |\\
&\quad +  |(H|C)_{s} \circ (H|C)_{t}^{-1}(z) -(H|C)_{s} \circ (H|C)_{t}^{-1}(w) |\\  
&\quad +  |(H|B)_{s} \circ (H|B)_{t}^{-1}(w) -(H|B)_{s} \circ (H|B)_{t}^{-1}(y) |\\
&\leq (1+\e)(|x-z| + |z-w| + |w-y|)\\
&= (1+\e)|x-y|.\qedhere
\end{align*}
\end{proof}

Our next result involves the removability of a Cantor set for a bi-Lipschitz path.

\begin{proposition}
\label{prop:remov}
Let $X \subset \R^2$ be a Cantor set.
For each $0\leq t\leq 1$, suppose $F_t:\R^2 \to \R^2$ is a continuous mapping such that $F|\R^2 \setminus X$ is a bi-Lipschitz path. Then $F$ extends to a bi-Lipschitz path on $\R^2$.
\end{proposition}

\begin{proof}
First, it is well-known that a bi-Lipschitz map on $U$ can be extended to a bi-Lipschitz map on the metric closure of $U$.
Hence the hypothesis that $F$ is a bi-Lipschitz path on $\R^2 \setminus X$, and Lemma \ref{lem:pathbound} imply that there exists $L\geq 1$ so that each $F_t$ is $L$-bi-Lipschitz on $\R^2$.

Next, we show that property (i) holds in the definition of a bi-Lipschitz path. Suppose $\epsilon >0$ is given and find $\delta >0$ so that if $|s-t| < \delta$ then for any $z \in \R^2 \setminus X$, $|F_t(z) - F_s(z)| <\epsilon /2$. If $x\in X$, find a sequence $(x_n)$ in $\R^2 \setminus X$ with $x_n \to x_0$. Then if we choose $N\in \N$ so that $|x-x_n| < \epsilon /4L$ for $n \geq N$, we have
\begin{align*}
|F_t(x) - F_s(x) | & \leq |F_t(x) - F_t(x_n)| + |F_t(x_n) - F_s(x_n)| + | F_s(x_n) - F_s(x)| \\
&\leq 2L|x-x_n| + \epsilon /2 \\
&< \epsilon .
\end{align*}
Hence condition (i) is satisfied.

Turning now to property (ii), let $x\in X$ and $y\in \R^2$ and find sequences $(x_n), (y_n)$ in $\R^2 \setminus X$ with $x_n \to x$ and $y_n \to y$. Given $\epsilon >0$, find $\delta >0$ so that if $|s-t| <\delta$ and $z,w \in \R^2 \setminus X$ then 
\[ |F_t(z) - F_t(w)| < (1+\epsilon / 3) | F_s(z) - F_s(w)|.\]
Next, choose $N\in \N$ large enough that if $n\geq N$ then $|x-x_n| < L^{-2}\epsilon |x-y| / 3$ and $|y-y_n| < \epsilon L^{-2} |x-y| / 3$. Hence
\[ |F_t(x) - F_t(x_n) | \leq L|x-x_n| < \epsilon(3L)^{-1} |x-y| \leq \epsilon|F_s(x) - F_s(y)|/3\]
and
\[ |F_t(y_n) - F_t(y)| \leq L | y_n-y| < \epsilon(3L)^{-1} |x-y| \leq \epsilon |F_s(x) - F_s(y)|/3.\]
It follows that
\begin{align*}
|F_t(x) - F_t(y)| &\leq |F_t(x) - F_t(x_n) | + |F_t(x_n) - F_t(y_n)| + |F_t(y_n) - F_t(y)| \\
&< ( 1+ \epsilon ) | F_s(x) - F_s(y) |.
\end{align*}
We conclude that property (ii) is satisfied and hence $F$ is a bi-Lipschitz path on $\R^2$.
\end{proof}

\subsection{Uniform families of bi-Lipschitz paths}

For the construction in the proof of Theorem \ref{thm:main1}, it will be useful to consider collections of bi-Lipschitz paths with uniform control. To that end we make the following definition.

\begin{definition}
\label{def:unifpath}
A collection $\mathcal{H}$ of bi-Lipschitz paths $H: [0,1] \rightarrow LIP(X)$ on a common metric space $X$ is a {\it uniform family of bi-Lipschitz paths} if
\begin{enumerate}[(i)]
\item there exists $L\geq 1$ so that $H_0$ has isometric distortion bounded above by $L$ for all $H\in \mathcal{H}$,
\item given $\epsilon >0$, there exists $\delta >0$ so that if $s,t\in [0,1]$ with $|s-t| <\delta$ then the two conditions in Definition \ref{bilippath} hold simultaneously for all $H \in \mathcal{H}$.
\end{enumerate}
\end{definition}

It is clear from Definition \ref{def:unifpath} and Lemma \ref{lem:pathbound} that there is a uniform bound on the isometric distortion of any map from any path in a uniform family of bi-Lipschitz paths. We have the following composition result.

\begin{lemma}
\label{lem:pathcomp}
Let $\mathcal{H}$ and $\mathcal{G}$ be two uniform families of bi-Lipschitz paths so that for each $G\in \mathcal{G}$ and each $t\in[0,1]$, $G_t$ is an isometry. Then the family $\mathcal{F} = \{ G\circ H : G\in \mathcal{G},H \in \mathcal{H} \}$ is a uniform family of bi-Lipschitz paths.
\end{lemma}

\begin{proof}
Given $\epsilon >0$, find $\delta >0$ so that both conditions in Definition \ref{bilippath} and Definition \ref{def:unifpath} hold for $|s-t| <\delta$, all $H\in \mathcal{H}$ and all $G \in \mathcal{G}$. Fixing $G \in \mathcal{G}$ and $H \in \mathcal{H}$, let $F  = G\circ H$. Then using the fact that $G_t$ is an isometry,
\begin{align*}
| F_t(z) - F_s(z) | &= |G_t(H_t(z)) - G_s(H_s(z)) | \\
&= | G_t(H_t(z)) - G_t(H_s(z)) + G_t(H_s(z)) - G_s(H_s(z)) | \\
&\leq | H_t(z) - H_s(z) | + |G_t(H_s(z)) -G_s(H_s(z)) | \\
& \leq 2 \epsilon,
\end{align*}
which verifies that condition (i) of Definition \ref{bilippath} holds uniformly for all paths in $\mathcal{F}$. For condition (ii), we have
\begin{align*}
|F_t(z) - F_t(w)| &= |G_t(H_t(z)) - G_t(H_t(w))| \\
&= |H_t(z) -H_t(w)| \\
& \leq (1+\epsilon) |H_s(z) - H_s(w)| \\
&= (1+\epsilon) |G_s(H_s(z)) - G_s(H_s(w))| \\
&= (1+\epsilon) |F_s(z) -F_s(w)|,
\end{align*}
which verifies that condition (ii) of Definition \ref{bilippath} holds uniformly for all paths in $\mathcal{F}$. Finally, since $G_0$ is a isometry,
and the isometric distortion of $H_0$ is uniformly bounded above, it follows that the same is true for any $F \in \mathcal{F}$.
\end{proof}

For our next result, we see that a family of conjugates of a bi-Lipschitz path by controlled dilations is uniform.

\begin{lemma}
\label{lem:pathdil}
Let $F:[0,1] \to LIP(\R^n)$ be a bi-Lipschitz path and let $c\geq 1$. The family 
\[\mathcal{F} = \{ \phi\circ F \circ \phi^{-1}  : \text{$\phi$ is a similarity of $\R^n$ with scaling factor at most $c$}\}\] 
is a uniform family of bi-Lipschitz paths.
\end{lemma}

\begin{proof}
Fix $\phi:\R^n \to\R^n$ to be a similarity of scaling factor $\lambda \leq c$.

First, suppose $F_0$ has isometric distortion $L_0$. Then we have
\[ |\phi \circ F_0 \circ \phi^{-1}(z) - \phi \circ F_0 \circ \phi^{-1}(w)| = \lambda | F_0\circ \phi^{-1}(z) - F_0\circ \phi^{-1}(w) |\]
from which it easily follows that $\phi \circ F_0 \circ \phi^{-1}$ is $L_0$-bi-Lipschitz.

Next, given $\epsilon >0$, find $\delta >0$ so that the two conditions in Definition \ref{bilippath} hold for $F_t$. 
If $s,t \in [0,1]$ with $|s-t |<\delta$, then
\[ |\phi \circ F_t \circ \phi^{-1}(z) - \phi \circ F_s \circ \phi^{-1}(z))| = \lambda | F_t\circ \phi^{-1}(z) - F_s\circ \phi^{-1}(z) | <\lambda \epsilon \leq c\epsilon\]
and we conclude property (i) of Definition \ref{bilippath} holds uniformly in $\mathcal{F}$. Finally, 
\begin{align*}
|\phi \circ F_t \circ \phi^{-1}(z) - \phi \circ F_t \circ \phi^{-1}(w)| &= \lambda | F_t \circ \phi^{-1}(z) - \circ F_t \circ \phi^{-1}(w) | \\
&\leq \lambda (1+\epsilon) | F_t \circ \phi^{-1}(z) - \circ F_t \circ \phi^{-1}(w) | \\
&= (1+\epsilon) |\phi \circ F_s \circ \phi^{-1}(z) - \phi \circ F_s \circ \phi^{-1}(w)|,
\end{align*}
from which we conclude condition (ii) in Definition \ref{bilippath} holds uniformly in $\mathcal{F}$.
\end{proof}

\subsection{Bi-Lipschitz paths on triangles}

As part of our construction, we will be using specific bi-Lipschitz paths which deform triangles in $\R^2$. Let $T$ be a triangle in $\C$. If the vertices are $w_1,w_2,w_3$, taken in counterclockwise order, then we may also denote this triangle by $T(w_1,w_2,w_3)$. In our construction, there will be two triangles $T_1$ and $T_2$ which share a vertex, and a bi-Lipschitz path $G: [0,1] \to LIP(\R^2)$ such that $G_0$ is the identity in $T_1$ and $G_1$ is an affine map from $T_1$ onto $T_2$. We will focus on constructing $G_t$ inside $T_1$.

After conjugating by an affine map, we may assume that $T_1$ and $T_2$ share $0$ as a vertex, that $T_1 = T(0,1,a)$ and $T_2 = T(0,c,b)$ for $0<\arg a < \pi$, $0\leq \arg c < \arg b \leq \pi$ and $ \arg b - \arg c < \pi$. These restrictions ensure that neither $T_1$ nor $T_2$ degenerate to line segments.

\begin{proposition}
\label{prop:tripath}
There exists a bi-Lipschitz path $G: [0,1]\to LIP( T_1)$ such that $G_0$ is the identity and $G_1$ is the map given by
\[ G_1(z) = \left ( \frac{ b-\overline{a}c}{a-\overline{a} } \right ) z + \left ( \frac{ ac-b}{a-\overline{a} } \right ) \overline{z}.\]
\end{proposition}

\begin{proof}
First, every real-linear map in $\C$ is of the form $Az+B\overline{z}$ for $A,B \in \C$, and since we require our maps to be orientation-preserving, we have $|A| > |B|$. Given $T_1 = T(0,1,a)$ and $T_2 = (0,c,b)$, it is elementary to check that
\[ g_1(z) = \left ( \frac{ b-\overline{a}c}{a-\overline{a} } \right ) z + \left ( \frac{ ac-b}{a-\overline{a} } \right ) \overline{z}\]
fixes $0$, maps $1$ to $c$ and maps $a$ to $b$.

Set $\gamma_1(t) = ct + (1-t)$ for $0\leq t \leq 1$ and $\gamma_2(t) = bt + (1-t)a$ for $0\leq t \leq 1$. Then define 
\[ G_t(z) = \left ( \frac{ \gamma_2(t) - \overline{a} \gamma_1(t) }{a-\overline{a} } \right )z + \left ( \frac{ a\gamma_1(t) - \gamma_2(t) }{a-\overline{a} }\right ) \overline{z} .\]
For any $t\in [0,1]$, $G_t$ maps $T_1$ onto the triangle with vertices $0, \gamma_1(t)$ and $\gamma_2(t)$.

For ease of notation, we define the map $h:= G_t\circ G_s^{-1}$, which maps $T(0,\gamma_1(s),\gamma_2(s))$ onto $T(0,\gamma_1(t) , \gamma_2(t))$. Since $h(z) = \gamma_1(t) \alpha ( z/\gamma_1(s))$, where $\alpha$ maps $T(0,1,\gamma_2(s)/\gamma_1(s))$ onto $T(0,1,\gamma_2(t) / \gamma_1(t) )$, we can compute that
\begin{align}
\label{eq:p1}
h(z) &= \left ( \frac{ \gamma_2(t) / \gamma_1(t) - \overline{ \gamma_2(s) / \gamma_1(s) } }{ \gamma_2(s) / \gamma_1(s) - \overline{ \gamma_2(s) / \gamma_1(s) } } \right ) \left ( \frac{ \gamma_1(t)z}{\gamma_1(s) } \right )\\
& \qquad+ \left ( \frac{ \gamma_2(s) / \gamma_1(s) - \gamma_2(t) / \gamma_1(t)  }{ \gamma_2(s) / \gamma_1(s) - \overline{ \gamma_2(s) / \gamma_1(s) } } \right  ) \left ( \frac{ \gamma_1(t) \overline{z}}{\overline{\gamma_1(s)}} \right ) \notag.
\end{align}

We collect some estimates that we will need. First we may suppose there exists $R>0$ so that 
\begin{equation}
\label{eq:p2}
\frac{1}{R} \leq | \gamma_i(t)| \leq R,
\end{equation}
for all $t\in [0,1]$ and $i\in \{1,2\}$.
Note also that $|z| \leq R$ for $z\in T(0,\gamma_1(s), \gamma_2(s))$.
Next, since the interior of every open triangle $G_t(T(0,1,a))$ is contained in the upper half-plane and the triangles do not degenerate, there exists $r>0$ so that
\begin{equation}
\label{eq:p3}
\sup_{t\in [0,1]} \ \left |\frac{\gamma_2(t)}{\gamma_1(t)} - \overline{ \left ( \frac{ \gamma_2(t)}{\gamma_1(t) } \right ) } \right | >r .
\end{equation}
We also observe that
\begin{equation}
\label{eq:p4}
|\gamma_1(t) - \gamma_1(s)| = |c-1| \cdot |t-s| , \quad |\gamma_2(t) - \gamma_2(s) | = |b-a| \cdot |t-s |,
\end{equation}
and hence that
\begin{equation}
\label{eq:p4a}
 \left | \frac{ \gamma_1(t) }{\gamma_1(s) } \right | \leq 1 + \frac{ |\gamma_1(t) - \gamma_1(s)| }{|\gamma_1(s)|} \leq 1+ R|c-1| |t-s|.
\end{equation}
Finally, via an elementary calculation we have
\[ \frac{ \gamma_2(s) }{ \gamma_1(s)} - \frac{\gamma_2(t)}{\gamma_1(t) } = \frac{(b-ac)(s-t)}{ (1+t(c-1))(1+s(c-1) )}\]
and hence by \eqref{eq:p2} we have
\begin{equation}
\label{eq:p6}
\left | \frac{ \gamma_2(s) }{ \gamma_1(s)} - \frac{\gamma_2(t)}{\gamma_1(t) }  \right | \leq R^2 |b-ac|  |s-t|.
\end{equation}

We can now prove property (i) for showing $G$ is a bi-Lipschitz path. Given $\epsilon >0$, we choose $\delta <  \epsilon / \xi$, with $\xi$ chosen below, so that if $|t-s| <\delta$ with $s,t \in [0,1]$, then by \eqref{eq:p1},
\begin{align*}
|h(z) - z| &= \left | z \left ( \frac{ \gamma_2(t) - \gamma_2(s) }{\gamma_1(s)} + \left ( \frac{\overline{\gamma_2(s) }}{|\gamma_1(s) |^2} \right ) \cdot \left (  \gamma_1(s) - \gamma_1(t)  \right ) \right ) \left (  \frac{\gamma_2(s)}{\gamma_1(s)} - \overline{\frac{\gamma_2(s)}{\gamma_1(s)} } \right )^{-1} \right .\\
&\qquad+ \left . \overline{z} \left ( \frac{ \gamma_2(s) / \gamma_1(s) - \gamma_2(t) / \gamma_1(t)  }{ \gamma_2(s) / \gamma_1(s) - \overline{ \gamma_2(s) / \gamma_1(s) } } \right  ) \left ( \frac{ \gamma_1(t) }{\overline{\gamma_1(s)}} \right ) \right |.
\end{align*}
Using \eqref{eq:p2}, \eqref{eq:p3} , \eqref{eq:p4} and \eqref{eq:p6}, we obtain
\[ |h(z)-z| \leq \frac{R}{r} \left (  R|b-a| |t-s| +  R^3|c-1| |t-s| \right ) + \frac{R^5|b-ac|}{r} |t-s|.\] 
By choosing 
\[ \xi = \frac{R^2 |b-a|}{r} + \frac{R^4 |c-1|}{r} + \frac{R^5|b-ac|}{r},\]
we obtain $|h(z) - z| <\epsilon$ for $z\in T(0,\gamma_1(s) , \gamma_2(s))$ as required.

Next, we prove property (ii). If $h(z) = Az+B\overline{z}$ is orientation preserving, then $|A| > |B|$ and $h$ is bi-Lipschitz with isometric distortion given by
\begin{equation}
\label{eq:bldist} 
\max \left \{ |A| + |B| , \frac{1}{|A|-|B| } \right \} .
\end{equation}
In our setting, 
\[ |A| = \left | \frac{ \gamma_1(t) }{\gamma_1(s)} \right | \cdot \left | \frac{\gamma_2(t)}{\gamma_1(t)} - \overline{ \frac{\gamma_2(s)}{\gamma_1(s) } } \right | \cdot \left | \frac{\gamma_2(s)}{\gamma_1(s)} - \overline{ \frac{ \gamma_2(s)}{\gamma_1(s)} } \right |^{-1}.\]
We can compute that
\[ \frac{ \gamma_2(t) / \gamma_1(t) - \overline{ \gamma_2(s)/\gamma_1(s)} }{ \gamma_2(s) /\gamma_1(s) - \overline{ \gamma_2(s) / \gamma_1(s)} } = 1 + \frac{ \gamma_2(t)/ \gamma_1(t) - \gamma_2(s)/ \gamma_1(s) }{  \gamma_2(s) /\gamma_1(s) - \overline{ \gamma_2(s) / \gamma_1(s)} }  \]
and hence by \eqref{eq:p4a} and \eqref{eq:p6} we obtain
\begin{equation}
\label{eq:p7}
|A| \leq \left  (1+ R|c-1| |t-s| \right ) \left ( 1+ \frac{R^2 |b-ac| }{r} |t-s| \right ).
\end{equation}
We also have that
\[ |B| = \left | \frac{ \gamma_1(t) }{\overline{\gamma_1(s)}} \right |  \cdot \left | \frac{\gamma_2(t)}{\gamma_1(t)} -  \frac{\gamma_2(s)}{\gamma_1(s) }  \right | \cdot \left | \frac{\gamma_2(s)}{\gamma_1(s)} - \overline{ \frac{ \gamma_2(s)}{\gamma_1(s)} } \right |^{-1}\]
and it follows from the comment after \eqref{eq:p2}, \eqref{eq:p3} and \eqref{eq:p6} that
\begin{equation}
\label{eq:p8}
|B| \leq \frac{R^4|b-ac|}{r} |s-t|.
\end{equation}
It then follows easily from \eqref{eq:bldist}, \eqref{eq:p7} and \eqref{eq:p8} that given $\epsilon >0$ we can choose $\delta >0$ so that if $s,t\in [0,1]$ with $|s-t| <\delta$, then $G_t\circ G_s^{-1}$ is $(1+\epsilon)$-bi-Lipschitz, completing the proof that $G_t$ is a bi-Lipschitz path.
\end{proof}

\subsection{Dehn twists and conjugates}
\label{sect:f}

Suppose $R\subset \R^2$ is a ring domain. Then we can define a Dehn twist $\mathfrak{D}$ in $R$ as follows. There exists $\eta \in (0,1)$ and a conformal map bijection $g:S\to R$, where $S = B(0,1)\setminus \overline{B}(0,1-\eta)$ is a round ring. By the conformal invariance of the modulus of ring domains, $\eta$ is uniquely defined. The Dehn twist in $S$ is given in polar coordinates by
\[ \mathfrak{D}(r,\theta) = \left ( r , \theta + 2\pi\tfrac1{\eta}(1-r) \right ),\]
and then the Dehn twist in $R$ is given by $g\circ \mathfrak{D} \circ g^{-1}$.

\begin{lemma}
\label{lem:dehntwistpath}
Let $\eta >0$, and consider $S$ and $\mathfrak{D}$ as above.
For $0\leq t\leq 1$ and $re^{i\theta} \in S$, set 
\[ D_t(r,\theta) = \left ( r , \theta + 2\pi t \tfrac1{\eta}(1-r) \right ).\]
Then $D$ is a bi-Lipschitz path in $S$ connecting the identity to $\mathfrak{D}$.
\end{lemma}

\begin{proof}
For convenience, for $1-\eta \leq r \leq 1$, set $h(r) =  \frac1{\eta}(1-r)$. Let $z = re^{i\theta} \in S$. If $s,t \in [0,1]$, we have
\begin{align*} 
|D_t(z) - D_s(z) | &= |ze^{2\pi i t h(|z|)} - ze^{2\pi i s h(|z|)}| = |z| \left | 1 - e^{2\pi i (s-t) h(|z|) } \right |.
\end{align*}
Since $|z| \leq 1$ and $h(|z|) \in [0,1]$, it is clear that the first condition in the definition of a bi-Lipschitz path is satisfied.

Next, it is clear that 
\[ D_t^{-1} ( re^{i\theta} ) = re^{ i ( \theta - 2\pi t h(r) ) },\]
and hence for $z,w\in S$ and $s,t \in [0,1]$ we have
\begin{align*}
| D_s \circ D_t^{-1} (z) - D_s \circ D_t^{-1} (w) |  &= | ze^{2\pi i h(|z|) (s-t)} - we^{2\pi i h(|w|) (s-t) }| \\
&= |z-we^{ 2\pi i (s-t)(h(|w|) - h(|z|) ) }|.
\end{align*}
Since $h$ is linear, it follows that
\[ | h(|w|) - h(|z|) |  = \tfrac1{\eta} \left | |w| - |z| \right| \leq \tfrac1{\eta}|w-z|.\]
We conclude that there exists $C>0$ independent of $z,w$ such that
\begin{align*} 
| D_s \circ D_t^{-1} (z) - D_s \circ D_t^{-1} (w) | &= |z-we^{ 2\pi i (s-t)(h(|w|) - h(|z|) ) }| \\
&\leq |z-w| + |w| \left | 1 - e^{ 2\pi i (s-t)(h(|w|) - h(|z|) ) } \right | \\
&\leq |z-w| ( 1+C|s-t| ),
\end{align*}
from which the second condition in the definition of a bi-Lipschitz path is satisfied.
\end{proof}

We will need to know that conformal conjugates of $D_t$ are also bi-Lipschitz paths. As was observed in \cite[Remark 2.6]{FM}, the conjugate of a bi-Lipschitz path on a closed manifold by a conformal map is a bi-Lipschitz path and this cannot be weakened to conjugation by a diffeomorphism. However, here we have closed ring domains, and so this remark does not immediately apply.

\begin{proposition}
\label{prop:path}
Suppose $S$ is a round annulus, $R$ is a ring domain with smooth boundary components, $g:S\to R$ is a conformal map and $F:[0,1] \to LIP(S)$ is a bi-Lipschitz path such that $F_t(S)= S$ for all $t\in [0,1]$. If $H = g \circ F \circ g^{-1}$, then $H$ is a bi-Lipschitz path with $H_t(R) = R$ for each $t \in [0,1]$.
\end{proposition}

\begin{proof}
We start with condition (i) from Definition \ref{bilippath}. 
Since the boundary components of $R$ are assumed smooth, the Riemann map $g$ and all its derivatives extend continuously to $\partial S$, see \cite[p.24]{BeKr}. In particular, there exists an upper bound $M$ for both $|g'|$ and $|(g^{-1})'|$ on $\overline{S}$ and $\overline{R}$, respectively. Hence $g$ and $g^{-1}$ are $M$-bi-Lipschitz maps.

Given $\epsilon >0$, find $\delta >0$ so that if $s,t \in [0,1]$ satisfy $|s-t| <\delta$ then 
\[ |F_s \circ F_t^{-1}(z) - z | < \epsilon / M\] 
for all $z \in R$. Then
\begin{align*}
|H_s \circ H_t^{-1}(z) - z | &= |g\circ F_s \circ F_t^{-1} \circ g^{-1}(z) - g(g^{-1}(z)) | \\
& \leq M |F_s \circ F_t^{-1} \circ g^{-1}(z) - g^{-1}(z) | \\
& \leq \epsilon.
\end{align*}
Hence condition (i) holds.

Next, for condition (ii), given $\epsilon >0$, find $\delta >0$ so that if $s,t \in [0,1]$ with 
\begin{equation}
\label{eq:path9}
|s-t| <\delta, \quad \text{ then } F_s\circ F_t^{-1} \text{ is $(1+\epsilon)$-bi-Lipschitz.}
\end{equation}

Consider the functions $p:\overline{S} \times \overline{S} \to \C$ and $q:\overline{R} \times \overline{R} \to \C$ defined by
\begin{equation}\label{eq:path1} 
p(z,w) = 
\begin{cases}
\frac{g(z) - g(w)}{z-w} - g'(w) &\text{ if $z\neq w$}\\
0 &\text{ if $z= w$}
\end{cases}
\end{equation}
and
\begin{equation}\label{eq:path1a} 
q(z,w) = 
\begin{cases}
\frac{g^{-1}(z) - g^{-1}(w)}{z-w} - (g^{-1})'(w) &\text{ if $z\neq w$}\\
0 &\text{ if $z= w$.}
\end{cases}
\end{equation}
By differentiability of $g$ and $g^{-1}$, both $p$ and $q$ are continuous functions on compact sets in $\C^2$ and hence bounded. That is, there exists $C>0$ so that $|p(z,w)|\leq C$ for all $(z,w) \in \overline{S} \times \overline{S}$, and $|q(z,w)| \leq C$, for all $(z,w) \in \overline{R} \times \overline{R}$.
%, where
%\begin{equation}
%\label{eq:path1a} 
%q(z,w) = \frac{ g^{-1}(z) - g^{-1}(w) }{z-w} - (g^{-1})'(w).
%\end{equation}
%Since $g$ is differentiable, we have $p(z,z) = 0$ for all $z\in \overline{S}$. 
Hence given $\epsilon '>0$, there exists $r>0$ so that if $z,w\in \overline{S}$ with $|z-w|<r$, then 
\begin{equation}
\label{eq:path10}
|p(z,w)| < \epsilon '. 
\end{equation}
By reducing $r$ if necessary, by the same reasoning we can also assume that if $z,w\in \overline{R}$ with $|z-w| <r$ then 
\begin{equation}
\label{eq:path11}
|q(z,w)| < \epsilon '.
\end{equation}

Now, let $z,w\in R$ with $|z-w|<r/[M(1+\epsilon)]$. Set $u= F_s \circ F_t^{-1}\circ g^{-1}(z)$ and $v= F_s \circ F_t^{-1}\circ g^{-1}(w)$. Then since $F_s\circ F_t^{-1}$ is $(1+\epsilon)$-bi-Lipschitz and $g^{-1}$ is $M$-bi-Lipschitz, we have $|u-v| <r$. Hence by \eqref{eq:path1} we have
\begin{equation}
\label{eq:path2}
|H_s\circ H_t^{-1} (z) - H_s\circ H_t^{-1} (w) | = |g(u)- g(v)| = | g'(v ) + p(u , v ) | \cdot |  u - v| .
\end{equation}
Next, again using the fact that $F_s\circ F_t^{-1}$ is $(1+\epsilon)$-bi-Lipschitz, we obtain
\begin{equation}
\label{eq:path3} 
|  u - v| \leq (1+\epsilon)| g^{-1}(z) - g^{-1}(w) |.
\end{equation}
Using \eqref{eq:path1a}, we have
\begin{equation}
\label{eq:path4}
|g^{-1}(z) - g^{-1}(w)| = | (g^{-1})'(w) + q(z,w) | \cdot |z-w|.
\end{equation}
Combining \eqref{eq:path2}, \eqref{eq:path3} and \eqref{eq:path4}, we obtain
\begin{equation}
\label{eq:path5} 
|H_s\circ H_t^{-1} (z) - H_s\circ H_t^{-1} (w) | \leq | g'(v ) + p(u , v ) | \cdot (1+\epsilon) \cdot | (g^{-1})'(w) + q(z,w) | \cdot |z-w|.
\end{equation}
Next, we have
\begin{align*}
|g'(v)&(g^{-1})'(w)|\\ 
&= | g'(g^{-1}(w))(g^{-1})'(w) + \left [ g'(F_s(F_t^{-1}(g^{-1}(w)) )) - g'(g^{-1}(w)) \right ] (g^{-1})'(w) |.
\end{align*}
Using condition (i) of $F$ being a bi-Lipschitz path, the fact that $|(g^{-1})'|$ is bounded and the fact that $g'$ is uniformly continuous on $\overline{S}$, by shrinking $\delta$ if necessary, we may conclude that 
\begin{equation}
\label{eq:path6}
|g'(v)(g^{-1})'(w)|  \leq 1+\epsilon.
\end{equation}
Combining \eqref{eq:path9}, \eqref{eq:path10}, \eqref{eq:path11}, \eqref{eq:path5}, \eqref{eq:path6},and  the bounds for the derivatives of $g$, $g^{-1}$ we obtain
\begin{align*}
|H_s\circ H_t^{-1} (z) - H_s\circ H_t^{-1} (w) |  &\leq (1+\epsilon)\left ( |g'(v)(g^{-1})'(w)| + |p(u,v)|(g^{-1})')(w) | \right . \\
& \left . \quad + \: |q(z,w) g'(v)| + |p(u,v)q(z,w)| \right ) |z-w|\\
&\leq (1+\epsilon) \left ( (1+\epsilon) + 2M\epsilon ' + (\epsilon ')^2 \right ) |z-w|.
\end{align*}

In particular, given $\eta >0$ we can find $\delta>0$ and $r>0$ so that if $s,t \in [0,1]$ with $|s-t| < \delta $, then for any $z,w \in \overline{R}$ with $|z-w| < r$ we have
\begin{equation}
\label{eq:path7}
|H_s\circ H_t^{-1} (z) - H_s\circ H_t^{-1} (w) | \leq (1+\eta) |z-w|.
\end{equation}

To show that condition (ii) holds, suppose for a contradiction that it does not. Then we can find $\eta >0$ and sequences $s_n,t_n$ in $[0,1]$ with $|s_n-t_n| \to 0$ and sequences $z_n,w_n$ in $\overline{R}$ for which
\begin{equation}
\label{eq:path8}
\left | \frac{ H_{s_n} \circ H_{t_n}^{-1} (z_n) -  H_{s_n} \circ H_{t_n}^{-1} (w_n) }{z_n-w_n} \right | > 1+\eta
\end{equation}
for all $n$.
By passing to subsequences, we may assume that $z_n \to z_0$ and $w_n \to w_0$. If $z_0 = w_0$ then we obtain a contradiction to \eqref{eq:path7}. Otherwise, suppose $|z_0 - w_0| = \xi$ and find $N\in N$ so that if $n\geq N$ then $|z_n - w_n| >\xi/2$.
By condition (i), we have
\begin{align*}
|H_{s_n} \circ H_{t_n}^{-1} (z_n) &-  H_{s_n} \circ H_{t_n}^{-1} (w_n) | \\
&\leq
|H_{s_n} \circ H_{t_n}^{-1} (z_n)  - z_n| + |z_n - w_n| + |H_{s_n} \circ H_{t_n}^{-1} (w_n) -w_n | \\
& \leq |z_n-w_n| + 2\epsilon \\
&\leq (1+\frac{4\epsilon}{\xi}) |z_n-w_n|.
\end{align*}
Since $|s_n-t_n| \to 0$, we can choose $n$ large enough so that $4\epsilon /\xi <\eta$ and hence contradict \eqref{eq:path8}.
We conclude that condition (ii) holds and hence $H_t$ is a bi-Lipschitz path.
\end{proof}

\subsection{Interpolation in an annulus}

In this subsection, we will prove the following interpolation result.

\begin{proposition}
\label{prop:interp}
Suppose $T>1$, and let $R=\{z\in \C : 1\leq |z| \leq T \}$ with boundary components $S_1 = \{z : |z|=1 \}$ and $S_T = \{ z : |z| = T \}$. Let $P: [0,1] \to LIP(S_1)$ and $Q:[0,1] \to LIP(S_T)$ be bi-Lipschitz paths such that $P_0 = P_1$ is the identity on $S_1$, $Q_0 = Q_1$ is the identity on $S_T$, and $\arg P_t, \arg Q_t$ are strictly increasing in $t$. Then there exists a bi-Lipschitz path $F:[0,1] \to LIP(R)$, with $F_0 =F_1$ the identity on $R$ and $F_t$ extends $P_t$ and $Q_t$ for each $t$.
\end{proposition}

%\vyron{why not just increasing?}

We start with the following fairly elementary estimate.

\begin{lemma}\label{lem:estimate}
Suppose that $c_1,c_2 >0$, $c_3\in [-3,3]$ and $a\in \R$. For any $\e>0$ and any $\d_1,\d_2 \in (-\e,\e)$, $\d_3\in (-2\e,2\e)$, we have
\[ |a+ i(c_1+c_2+c_3x) + i(\d_1c_1+\d_2c_2+\d_3a)| \leq (1+8\e)|a+ i(c_1+c_2+c_3a)|.\]
\end{lemma}

\begin{proof}
We consider three cases.

\emph{Case 1: $ac_3\geq 0$.} Then,
\begin{align*} 
 |a+ i(c_1+c_2+c_3a) &+ i(\d_1c_1+\d_2c_2+\d_3a)|\\ 
 &\leq |a+ i(c_1+c_2+c_3a)| + |\d_1|c_1+|\d_2|c_2+|\d_3||a|\\
 &\leq |a+ i(c_1+c_2+c_3a)| +(2\e)|i(c_1+c_2+c_3a)| + 2\e|a|\\
 &\leq (1+4\e)|a+ i(c_1+c_2+c_3a)|
\end{align*}

\emph{Case 2: $ac_3< 0$ and $c_1+c_2+c_3a\leq \frac12(c_1+c_2)$.} We have that 
\[ |a| > \frac1{2|c_3|}(c_1+c_2) \geq \frac16(c_1+c_2).\] 
Therefore,
\begin{align*}
|a+ i(c_1+c_2+c_3a) &+ i(\d_1c_1+\d_2c_2+\d_3a)|\\
&\leq |a+ i(c_1+c_2+c_3a)| + |\d_1|c_1+|\d_2|c_2+|\d_3||a|\\
&\leq |a+ i(c_1+c_2+c_3a)| + \e(c_1+c_2) + 2\e |a| \\
&\leq  |a+ i(c_1+c_2+c_3a)| + 8\e|a|\\
&\leq (1+8\e)|a+ i(c_1+c_2+c_3a)|.
\end{align*}

\emph{Case 3: $ac_3< 0$ and $c_1+c_2+c_3a> \frac12(c_1+c_2)$.} We have that 
\begin{align*}
|a+ i(c_1+c_2+c_3a) &+ i(\d_1c_1+\d_2c_2+\d_3a)|\\
&\leq |a+ i(c_1+c_2+c_3a)| + |\d_1|c_1+|\d_2|c_2+|\d_3||a|\\
&\leq |a+ i(c_1+c_2+c_3a)| + \e(c_1+c_2) + 2\e|a|\\
&\leq |a+ i(c_1+c_2+c_3a)| + 2\e|i(c_1+c_2+c_3a)| + 2\e|a|\\
&\leq (1+4\e)|a+ i(c_1+c_2+c_3a)|. \qedhere
\end{align*}
\end{proof}

Next, we prove an interpolation result on strips.

\begin{lemma}
\label{lemma:realpath}
Suppose that $F,G : [0,1] \to LIP(\R)$ be bi-Lipschitz paths with $F_0 (x)= G_0(x) = x$ for all $x\in \R$, $F_1(x) = G_1(x) = x+2\pi$ for all $x\in \R$, $F_t,G_t$ are $2\pi$-periodic for all $t\in [0,1]$ and $F_t(x), G_t(x)$ are both strictly increasing in $t$ for a fixed $x$. Let $M>0$ and let $S$ be the strip $S=\{ z\in \C: 0< \Re(z) < M\}$. Then there exists a bi-Lipschitz path $H:[0,1] \to LIP(\overline{S})$ which extends to $\partial S$ with $H_t(iy) = iF_t(y)$ and $H_t(M+iy) = M+iG_t(y)$ for $0\leq t\leq 1$ and $y\in \R$. Moreover, $H_0(z) = z$ and $H_1(z) = z+2\pi$ for all $z\in S$.
\end{lemma}

\begin{proof}
We define $H_t$ via the obvious convex interpolation in $S$. That is, we set
\[H_t(x+iy) = x + i \left ( G_t(y) + ( 1- x/M) ( F_t(y) - G_t(y) ) \right ) \]
for $0\leq t\leq 1$, $0\leq x\leq M$ and $y\in \R$. Clearly $H_0$ is the identity and $H_1$ is a translation by $2\pi i$. We need to show that $H_t$ is a bi-Lipschitz path.

We start by showing that each $H_t$ is a bi-Lipschitz map. Using Lemma \ref{lem:pathbound}, suppose that $F_t$ is $L$-bi-Lipschitz and $G_t$ is $\lambda$-bi-Lipschitz for all $t\in [0,1]$. 
Setting $z=x+iy$ and $w=x'+iy'$, we have
\begin{align*}
&|H_t(z) -H_t(w)|\\ 
&\leq |x-x'| + \left | \tfrac{x}{M} G_t(y) - \tfrac{x'}{M} G_t(y') + (1-\tfrac{x}{M} ) F_t(y) - (1-\tfrac{x'}{M} F_t(y') \right | \\
&\leq |x-x'| +(1-\tfrac{x}{M} ) |F_t(y) - F_t(y')| + \tfrac{x}{M} | G_t(y) - G_t(y') | + \frac{|x-x'|}{M} |G_t(y') - F_t(y')|\\
&\leq (1+2\pi) |x-x'| + \max\{ L , \lambda \} |y-y'| \\
&\leq \max\{ L,\lambda, 1+2\pi \} ( |x-x'| + |y-y'| )\\
&\leq 2\max\{ L,\lambda, 1+2\pi \} |z-w|.
\end{align*}
For the lower bound, we consider two cases. First, set \lefteqn{C = \min \left \{  \frac{M}{8\pi L} , \frac{M}{8\pi \lambda} , \frac12 \right \} .}

\emph{Case 1.} Suppose that $|x-x'| \geq C|y-y'|$. It follows that 
\begin{align*} 
|H_t(z) - H_t(w)| &\geq |x-x'| \geq \frac{C}{2} (|x-x'| + |y-y'|)\geq  \frac{C}{2} |z-w|.
\end{align*}

\emph{Case 2.} Suppose that $|x-x'| < C |y-y'|$. Without loss of generality, assume that $y' \leq y$. Then, 
\begin{align*}
|H_t&(z) - H_t(w)|\\
&\geq |(1-\tfrac{x}{M})(F_t(y)-F_t(y')) + \tfrac{x}{M}(G_t(y)-G_t(y')) + (G_t(y')-F_t(y'))(x-x')/M|\\
&\geq (1-\tfrac{x}{M})(F_t(y)-F_t(y')) + \tfrac{x}{M}(G_t(y)-G_t(y')) - |G_t(y')-F_t(y')||x-x'|/M\\
&\geq \min\{L^{-1},\lambda^{-1}\}|y-y'| - 2\pi |x-x'|/M \\
&\geq \left ( \min\{L^{-1},\lambda^{-1}\} - \frac{2\pi C}{M} \right ) |y-y'|\\
&\geq \frac{2\pi C}{M} |y-y'|\\
&\geq \frac{\pi C}{M} ( |x-x'| + |y-y'| )\\
&\geq \frac{\pi C}{M}|z-w|.
\end{align*}

Next, we show that $h_t$ satisfies condition (i) of Definition \ref{bilippath}. From Definition \ref{bilippath} (i), by setting $u=F_t^{-1}(x)$, it follows that given $\epsilon >0$, we may find $\delta> 0$ so that if $|s-t|<\delta$, then $|F_s(u) - F_t(u) | <\epsilon$ for all $u\in \R$. The same holds true for $G_t$. Now suppose $z\in S$ and $h_t^{-1}(z) = x+iy$. Then we have
\[ z = x + i\left ( G_t(y) + (1-\tfrac{x}{M}) (F_t(y ) - G_t(y) ) \right ) \]
and 
\[ H_s \circ H_t^{-1}(z) =x +  i\left ( G_s(y) + (1-\tfrac{x}{M}) (F_s(y ) - G_s(y) ) \right ).\]
Therefore, we obtain
\[ |H_s\circ H_t^{-1}(z) - z | = \left |x (G_t(y) - G_s(y)) + (1- \tfrac{x}{M}) (F_s(y) - F_t(y) ) \right | < \epsilon. \]
Hence condition (i) of Definition \ref{bilippath} is satisfied.

Finally, we show that $h_t$ satisfies condition (ii) of Definition \ref{bilippath}. Note that
\begin{align*} 
&H_s(z) - H_s(w)\\ 
&= (x-x') + i \left[ (1-\tfrac{x}{M})F_s(y) - (1-\tfrac{x'}{M})F_s(y') \right] + i \left[ xG_s(y) - x'G_s(y') \right]/M\\
&= (x-x') + i [ (1-\tfrac{x}{M})(F_s(y)-F_s(y')) + \tfrac{x}{M}(G_s(y)-G_s(y'))\\ 
&\quad+ (G_s(y')-F_s(y'))(x-x')/M].
\end{align*}
Fix $\e>0$. We know that there exists $\d>0$ such that if $|t-s| <\d$, then
\begin{align*}
|F_s(y)-F_s(y')| &\leq (1+\e)|F_t(y)-F_t(y')|,\\ 
|G_s(y)-G_s(y')| &\leq (1+\e)|G_t(y)-G_t(y')|
\end{align*}
and
\begin{align*}
|F_s(y)-F_t(y)| \leq \e, \quad |G_s(y)-G_t(y)| \leq \e.
\end{align*}

Therefore,
\begin{align*} 
H_s(z) &- H_s(w)\\ 
&= (x-x') + i [ (1-\tfrac{x}{M})(F_t(y)-F_t(y'))(1+\d_1) + \tfrac{x}{M}(G_t(y)-G_t(y'))(1+\d_2)\\ 
&\quad+ (G_t(y')-F_t(y') +\d_3)(x-x')/M ]\\
&= a + i(c_1 + c_2 + c_3a + \d_1 c_1 + \d_2c_2 + \d_3a)\\
&= a + i(c_1 + c_2 + c_3a) + (\d_1 c_1 + \d_2c_2 + \d_3a)i\\
&= H_t(z)-H_t(w) +  (\d_1 c_1 + \d_2c_2 + \d_3a)i
\end{align*}
where
\begin{align*}
a&= x-x'\\
c_1&= (1-\tfrac{x}{M})(F_t(y)-F_t(y'))\\
c_2&= x(G_t(y)-G_t(y'))/M\\
c_3&= (G_t(y')-F_t(y'))/M
\end{align*}
and $\d_1,\d_2,\d_3$ are functions of $x,y,x',y',s,t$ satisfying
\begin{align*}
|\d_1|<\e, \quad |\d_2|<\e, \quad |\d_3|< 2\e.
\end{align*}
Now, it follows from Lemma \ref{lem:estimate} that
\[ |H_s(z) - H_s(w)| \leq (1+8\e)|H_t(z)-H_t(w)|\]
and condition (ii) of Definition \ref{bilippath} is satisfied.
\end{proof}

We are now in a position to prove Proposition \ref{prop:interp}.

\begin{proof}[Proof of Proposition \ref{prop:interp}]
The idea is to lift via the exponential function and then use Lemma \ref{lemma:realpath}. To that end, define $\widetilde{P}$ and $\widetilde{Q}$ via the functional equations $P\circ \exp = \exp \circ \widetilde{P}$ and $Q\circ \exp = \exp \circ \widetilde{Q}$. 

Since the exponential function is conformal and has uniformly bounded derivative on the strip $S = \{ z : 0 \leq \Re(z) \leq \ln T \}$, we conclude via the same argument as in Proposition \ref{prop:path} that $\widetilde{P}$ and $\widetilde{Q}$ are bi-Lipschitz paths in the lines $\{z : \Re(z) = 0\}$ and $\{z : \Re(z) = \ln T \}$ respectively.

Applying Lemma \ref{lemma:realpath} to the strip $S = \{ z : 0 \leq \Re(z) \leq \ln T\}$ with boundary bi-Lipschitz path $\widetilde{P}$ and $\widetilde{Q}$, we obtain a bi-Lipschitz path $\widetilde{F}$ which extends the boundary bi-Lipschitz paths.

Since $\widetilde{F_t}$ is $2\pi$-periodic by construction, we obtain the required bi-Lipschitz path $F$ via $F \circ \exp = \exp \circ \widetilde{F}$, again using the fact that the exponential function has uniformly bounded derivative in $S$.
\end{proof}

\section{Bi-Lipschitz collapsing for sets of small Assouad dimension}\label{sec:collapse}

%In what follows, we shall state some of our set-up for $\R^2$, but our applications will be purely two dimensional.

The goal in this section is to show that for a Cantor set $X\subset \R^2$ with $\dim_A X<1$, we can cover it by small topological disks that can then by collapsed via a bi-Lipschitz path into a small disk. This is the content of Proposition \ref{prop:1} below, see Figure \ref{fig:5} for a schematic.

\begin{figure}[h]
\begin{center}
\includegraphics[width=4in]{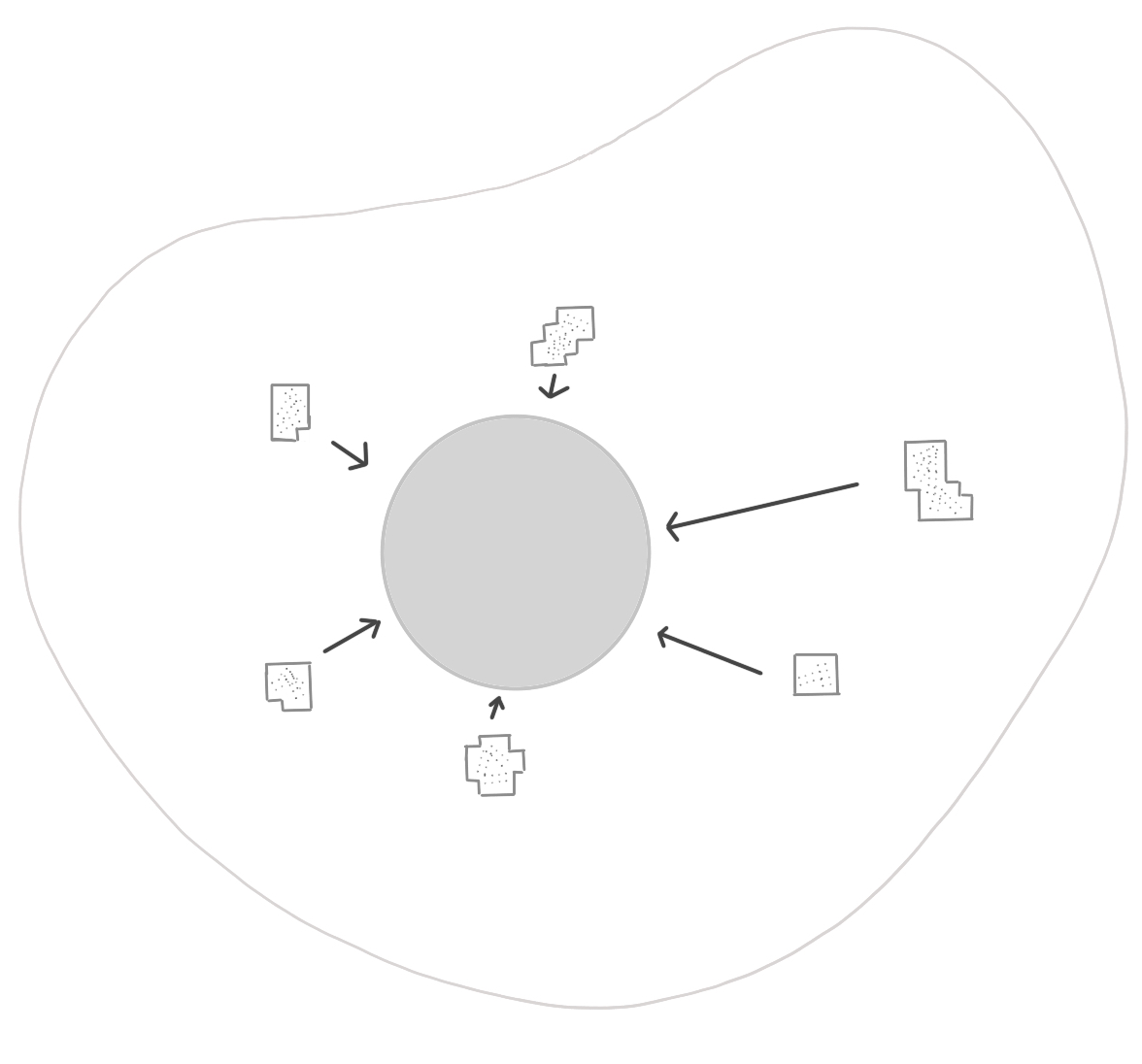}
\caption{The larger domain is $\Omega$, the shaded ball is $B$, the PL curves give the boundaries of the components of $\mathcal{T}_{\d}(X)$ and the arrows indicate that the bi-Lipschitz path $H_t$ constructed in Proposition \ref{prop:1} moves these components into $B$ in an isometric way.}
\label{fig:5}
\end{center}
\end{figure}

\begin{proposition}\label{prop:1}
Let $C>0$, $c\geq 1$, $s\in [0,1)$, $\eta \in (0,1)$ and let $\Omega \subset \R^2$ be a domain with $\diam \Omega =1$ such that for any $x,y \in \Omega$ there exists a path $\g_{x,y} : [0,1] \to \Omega$ such that $\g_{x,y}(0) = x$, $\g_{x,y}(1)=y$ and 
\begin{equation}\label{eq:cigar} 
\dist(\g_{x,y},\partial U) \geq (2c)^{-1}\min\{\dist(x,\partial \Omega), \dist(y,\partial \Omega)\}.
\end{equation}
Let $X\subset \Omega$ be $(C,s)$-homogeneous with $\dist(X,\partial\Omega) > \eta $.
There exists $\e >0$ so that if $z,w \in \Omega$ have distance at least $2\epsilon$ from $\partial \Omega$, then the disk $B(z,\e)$ can be deformed continuously and isometrically to $B(w,\e)$ in $\Omega$. There exists $\d >0$ so that if 
$B = B(z,\e) \subset \Omega$ is a disk of radius $\e$ with center $z$ satisfying $\dist (z,\partial \Omega)  \geq 2\e$, there exists a bi-Lipschitz path $H:[0,1] \to LIP(\Omega)$ such that
\begin{enumerate}
%\item $X\subset \bigcup_{i=1}^n D_i$ and domains $D_i$ have mutually disjoint closures;
%\item for any $i\in\{1,\dots,n\}$ and any $x\in \partial D_i$, $\dist(z,X) \simeq \d\diam{X}$;
\item $H_1$ maps the closed neighborhood $\mathcal{T}_{\d}(X)$  of $X$ into $B$;
\item for each $t\in [0,1]$ and each component $D$ of $\mathcal{T}_{\d}(X)$, the map $H_t|D$ is an isometry. 
\end{enumerate}
\end{proposition}

A couple of remarks are in order.

\begin{remark}
First, condition (\ref{eq:cigar}) on $\Omega$ is inspired by, but slightly weaker than, the well-known $c$-John property. Second, for the rest of this section, we call curves $\g_{x,y}$ $c$-cigar curves. Finally, if $c'>c$, then there exists a piecewise linear (abbv. PL) $c'$-cigar curve $\sigma$ joining $x$ with $y$ in $\Omega$. In light of this observation, we will assume from now on that all cigar curves are PL.
\end{remark}

\subsection{Convex sets}

Given a set $E \subset \R^N$, we denote by $\Hull(E)$ the \emph{closed convex hull} of $E$, that is, the intersection of all closed convex sets that contain $E$. Such a set is itself convex and $\diam(\Hull(E)) = \diam(E)$. 

\begin{lemma}
Let $E\subset \R^N$ be a bounded set. If $x,y \in \Hull(E)$ and $|x-y| = \diam(\Hull(E))$, then $x,y \in \overline{E}$. 
\end{lemma}

\begin{proof}
For a contradiction, assume that $x$ is not in $\overline{E}$. That is, $r:=\dist(x,\overline{E})>0$. Let $P \subset \R^N$ be the $(N-1)$-plane that contains $x$ and is orthogonal to the line segment $[x,y]$. Then, since $|x-y| = \diam(\Hull(E))$, it follows that $\Hull(E)$ lies on $\overline{H}$ where $H$ is one of the two components of $\R^N \setminus P$. Therefore, 
\[ E \subset (H \cap B(y,|x-y|) )\setminus B(x,r).\]
Then, setting $\d = \dist( \partial B(x,r) \cap \partial B(y,|x-y|), P)$ we have that the set
\[ \{z\in H : \dist(z,P)\geq \d\}\cap \text{Hull}(E)\]
is a convex set which contains $E$ and is a proper subset of $\Hull(E)$, which is a contradiction.
\end{proof}

\begin{lemma}\label{lem:convexsub}
Let $E_1,\dots,E_n$ be sets in $\R^N$. There exists $l\in\{1,\dots,n\}$ and there exist mutually disjoint convex closed sets $\D_1,\dots,\D_l$ in $\R^N$ 
such that each $E_i$ is contained in some $\D_j$ and
%with the following properties. For each $i\in\{1,\dots,n\}$ there exists $j\in\{1,\dots,l\}$ such that $E_i\subset H_j$, and
\[\sum_{j=1}^l \diam{\D_j} \leq \sum_{i=1}^n\diam{E_i}.\]
\end{lemma}

\begin{proof}
If one of the sets $E_i$ is unbounded, then set $l=1$, $\D_1=\R^N$ and the claim is trivial.

Assume now that all sets $E_i$ are bounded. In this case, the construction of the convex sets $\D_j$ is in an inductive fashion.

\emph{Step 1.} For each $i\in\{1,\dots,n\}$, let $\D_i^{(1)} = \Hull(E_i)$. If the sets $\D_i^{(1)}$ are mutually disjoint, then set $\D_i = \D_i^{(1)}$ and the procedure terminates; if some intersect, proceed to the next step.

\emph{Inductive Step.}
% Note that some of the $H_i^{(1)}$ may intersect. We modify now the convex hulls in an inductive fashion.
Suppose that for some $k\in\{1,\dots,n-1\}$ we have defined closed convex sets $\D_1^{(k)},\dots,\D_{n-k+1}^{(k)}$ such that at least two of them intersect. In particular, let $1 \leq i_0 < j_0 \leq n-1$ be such that $\D_{i_0}^{(k)} \cap \D_{j_0}^{(k)} \neq \emptyset$.
% be the smallest element in $\{1,\dots,n-k+1\}$ such that $H_{i_0}^{(k)}$ intersects with $\bigcup_{i\neq i_0}H_{i}^{(k)}$. Let also $j_0$ be the smallest element $j$ in $\{i_0+1,\dots,n-k+1\}$ such that $H^{(k)}_{j}$ intersects with $H^{(k)}_{i_0}$. 
 We now define $\D^{(k+1)}_i$ for $i\in \{1,\dots,n-k\}$ as follows:
\begin{itemize}
\item if $i < i_0$ or if $i_0<i<j_0$, then set $\D^{(k+1)}_i = \D^{(k)}_i$;
\item if $i=i_0$, then set $\D^{(k+1)}_{i_0}  = \Hull(\D^{(k)}_{i_0} \cup \D^{(k)}_{j_0})$;
\item if $j_0 \leq i \leq n-k$, then set $\D^{(k+1)}_i = \D^{(k)}_{i+1}$.
\end{itemize}
Note that
\[ \diam{\D^{(k+1)}_{i_0}} = \diam(\D^{(k)}_{i_0} \cup \D^{(k)}_{j_0}) \leq \diam{\D^{(k)}_{i_0}} + \diam{\D^{(k)}_{j_0}}.\]
If the sets $\D_i^{(k+1)}$ are mutually disjoint, then set $\D_i = \D_i^{(k+1)}$ and the procedure terminates; if some intersect, proceed to the next step.

It is clear that the procedure above will terminate in $m$ steps for some $m\in\{1,\dots,n\}$. The sets $\D_1,\dots,\D_{n-m+1}$ produced are convex, mutually disjoint, and each $E_i$ is contained in some $\D_j$. It remains to show that
\begin{equation}\label{eq:10}
\sum_{i=1}^{n-m+1} \diam{\D_i} \leq \sum_{i=1}^n\diam{E_i}.
\end{equation}
To prove (\ref{eq:10}), first note that for all $i\in\{1,\dots,n\}$, $\diam{E_i} = \diam{\D_i^{(1)}}$. Therefore, if $m=1$, then (\ref{eq:10}) follows. 

Suppose now that $m\geq 2$. Fix $k\in \{1,\dots,m-1\}$ and let $i_0,j_0 \in \{1,\dots,n-k+1\}$ be as in the construction of domains $\D_{i}^{(k+1)}$. Then,
\begin{align*}
\sum_{i=1}^{n-k+1} \diam{\D_i^{(k)}} &= \sum_{i \in\{1,\dots,n-k+1\}\setminus\{i_0,j_0\}} \diam{\D_i^{(k)}} + \diam{\D_{i_0}^{(k)}} + \diam{\D_{j_0}^{(k)}}\\
&\geq \sum_{i \in\{1,\dots,n-k\}\setminus\{i_0\}} \diam{\D_i^{(k+1)}} +  \diam{\D_{i_0}^{(k+1)}}\\
&= \sum_{i=1}^{n-k} \diam{\D_i^{(k+1)}}. 
\end{align*}
Now by induction, (\ref{eq:10}) follows.
\end{proof}

\begin{comment}
Given a convex set $\D \subset \R^N$, a PL curve $\g:[0,1] \to \R^N$ with $\g(0) \in \D$, and a number $r>0$ we denote
\[ \mathcal{N}(H,\g,r) := \bigcup_{t\in[0,1]}\bigcup_{x\in H} \overline{B}(x+\g(t),r).\]
\end{comment}

\begin{lemma}\label{lem:1}
Let $\Omega \subset\R^2$ be a domain with non-empty boundary, and let $\D \subset \Omega$ be a compact convex set with PL boundary. Let $\delta\in (0,1)$, let 
\[ 0 < r< (1-\delta) \operatorname{dist}(\D , \partial \Omega), \] 
and let $\g:[0,1] \to \Omega$ be a PL curve in $\Omega$ with $\g(0)\in \D$ and $|\g(t) - \g(0)| <r$ for all $t\in [0,1]$. Then there exists a bi-Lipschitz path $H:[0,1] \to LIP(\Omega)$ such that
\begin{enumerate}
\item for each $t\in[0,1]$, $H_t| \partial \Omega$ is the identity;
\item for each $t\in[0,1]$, $H_t| \D$ is a translation mapping with $H_t(\g(0))=\g(t)$.
\end{enumerate}
\end{lemma}

\begin{proof}
Without loss of generality, we may assume that $\g$ is a straight line segment; in the general case of PL curves $\g$, concatenate the bi-Lipschitz paths from the various segments of $\g$ and re-parameterize if necessary. Assume then, that $\g:[0,1] \to \Omega$ with $\g(t) = \g(0) + tv$ for some $v \in \C$ with $|v|<r$. 

By the hypotheses, $\D$ is a convex polygon with vertices $v_1,\ldots, v_n$. 
Fix $z_0 \in \D$ and for $i \in \{1,\ldots, n \}$ let $w_i$ be the point on the ray from $z_0$ through $v_i$ that is distance $(1-\delta /2)\dist (\D,\partial \Omega)$ away from $v_i$ (and outside $\D$). Let $Y$ be the convex hull of $w_1,\ldots, w_n$ and set $d = \dist (\D, \partial Y)>0$.

Triangulate the PL ring domain $\overline{Y\setminus \D}$ via triangles $T_1,\ldots, T_m$ which have, alternately, one or two vertices contained in $\partial Y$.

Given a direction $e^{i\theta}$, we will construct a bi-Lipschitz path which moves $\D$ onto $\D_1 = \{z : z = z' + de^{i\theta}/2, z' \in \D \}$. For $z \in \D$ we just define
\[ H_t(z) = (z+de^{i\theta}/2 ) t + (1-t)z.\]
If $T_i$ has two vertices on $\partial Y$ and third vertex $\xi_1\in \partial \D$, then we apply the bi-Lipschitz path from Proposition \ref{prop:tripath} (conjugated by a suitable similarity) which fixes the two vertices in $\partial Y$ and moves $\xi_1$ to $\xi_1 + de^{i\theta}/2 \in Y$.

If $T_i$ has one vertex on $\partial Y$ and two vertices $\xi_1,\xi_2$ in $\partial X$, then we apply the bi-Lipschitz path from Proposition \ref{prop:tripath} (again conjugated by a suitable similarity) which fixes the vertex in $\partial Y$ and moves $\xi_j$ to $\xi_j + de^{i\theta}/2$ for $j=1,2$.

This piecewise construction yields a bi-Lipschitz path which moves $\D$ to $\D_1$ and fixes every point of $\partial Y$ and hence can be extended to fix every point of $\Omega \setminus Y$.
By concatenating a finite number of bi-Lipschitz paths, we may move $X$ along any PL path in $\Omega$, as long as we avoid $\partial \Omega$, such that the path acts as a translation on $X$. \end{proof}

For the rest of the paper, given a bounded set $X \subset \R^2$, a number $r>0$ and a curve $\g :[0,a] \to \R^2$ with $\g(0) \in X$, we denote
\[ \mathcal{N}(X,\g, r) := \bigcup_{t\in[0,a]} \left(\g(t)-\g(0) + N(X,r)\right).\]
%\[ \mathcal{N}(X,\g, r) := \{ z \in \R^2 : \text{ there exists $t\in [0,1]$ such that $z+\g_m(0)-\g_m(t) \in N(X,r)$}\}.\]

\subsection{Proof of Proposition \ref{prop:1}}

The first claim about the existence of such an $\e$ follows by following a $c$-cigar curve from $z$ to $w$. Henceforth, fix $B = B(z_0,\e)$.

Suppose first that $\diam{X}=0$, that is $X=\{x_0\}$ for some $x_0\in \Omega$. Let $\g$ be a PL $c$-cigar path that joins $x_0$ with $x_0$ in $\Omega$. Let $\Delta$ be a compact convex set with PL boundary contained in $B(x_0,r)$ with $r < \min \{ \epsilon,  \tfrac13\dist(x_0,\partial\Omega) \}$. We then apply Lemma \ref{lem:1} to find the required bi-Lipschitz path $H:[0,1] \to LIP(\Omega)$ such that for any $t\in [0,1]$, $H_t(x_0) = \g(t)$.

Suppose now and for the rest of the proof of Proposition \ref{prop:1} that $\diam{X} > 0$.
Set 
\begin{equation}\label{eq:delta} 
\d = \left(\frac{\min\{\eta,\e\}}{216cC}\right)^{\frac1{1-s}}.
\end{equation}
We may assume that $C>1$, hence $\d$ is less than $1$. Then let $V$ be a $\d$-net of $X$ and let $D_1,\dots,D_n$ be the components of $\mathcal{T}_{\d}(X)$. 

Since $\d < \eta/20$, we have that 
\[ \dist(\mathcal{T}_{\d}(X),\partial\Omega) \geq \dist(X,\partial\Omega) - \dist_H(\mathcal{T}_{\d}(X),X) \geq \eta - 8\d \geq \eta/2\]
where $\dist_H$ denotes the Hausdorff distance.

Let $i\in\{1,\dots,n\}$. For each $x\in \partial D_i$ there exists $z\in X$ such that $|x-z|\leq 8\d$ and there exists $v\in V$ such that $|z-v|\leq \d$. Therefore, for every $x\in \partial D_i$, $\dist(x,V) \leq 9\d$ and it follows that
\begin{equation}\label{eq:3}
\diam{D_i} \leq 18\d\card(V\cap D_i).
\end{equation}
Therefore,
\begin{equation}\label{eq:1}
\sum_{i=1}^n\diam{D_i} \leq  18\d\card(V) \leq 18C\d^{1-s} = (12c)^{-1}\min\{\eta,\e\}.
\end{equation}

The construction of the bi-Lipschitz path $H$ consists of two parts. In the first part we construct at most $n-1$ many bi-Lipschitz paths that "gather the sets $D_i$ together" and in the second part we construct a bi-Lipschitz path that leads the cluster of gathered sets $D_i$ into the disk $B$.

\subsubsection{Part 1}

The construction in this part is in an inductive manner.

\emph{Step 0.} Apply Lemma \ref{lem:convexsub} for the sets $D_1,\dots,D_n$ and obtain closed mutually disjoint convex sets $\D_1^{(1)},\dots,\D_{k_0}^{(0)}$ for some positive integer $k_0\in \{1,\dots,n\}$. Note that 
\begin{align*} 
\sum_{i=1}^{k_0} \diam{\D^{(0)}_i} \leq \sum_{i=1}^{n} \diam{D_i} \leq (12c)^{-1}\min\{\eta,\e\}.
\end{align*}
Moreover, the sets $\D_1^{(0)},\dots,\D_{k_0}^{(0)}$ are contained in $\Omega$ and for each $i\in\{1,\dots,k_0\}$
\begin{align*} 
\dist(\D_i^{(0)},\partial\Omega) \geq \dist(\mathcal{T}_{\d}(X),\partial\Omega) - \diam{\D_i^{(0)}} \geq \eta/2 - (12c)^{-1}\eta >\tfrac13 \eta.
\end{align*}
If $k_0 = 1$, then the procedure terminates and we proceed to Part 2; otherwise proceed to the next step.

\emph{Inductive step.} Suppose that for some positive integer $m \in \{0,\dots,n-2\}$ we have defined disjoint closed convex sets $\D^{(m-1)}_{1},\dots, \D^{(m-1)}_{k_{m-1}} \subset \Omega$ such that $2 \leq k_{m-1} \leq n-m+1$ and the following three properties hold.
\begin{enumerate}
\item[(P1)] For each $i\in \{1,\dots,k_{m-1}\}$ there exists $j\in\{1,\dots,n\}$ with $D_{j} \subset \D^{(m-1)}_{i}$;
\item[(P2)] We have \[ \sum_{i=1}^{k_{m-1}} \diam{\D^{(m-1)}_{i}} < (6c)^{-1}\min\{\eta,\e\};\]
\item[(P3)]  For each $i\in \{1,\dots, k_{m-1}\}$, $\dist(\D_i^{(m-1)},\partial\Omega) > \eta/3$.
%\sum_{i=1}^{n} \diam{D_i} + (\textcolor{red}{??})
%\[ \dist(H_i^{(m-1)},\partial\Omega) \geq (\eta - 5\sqrt{N}\d) - (1+5\sqrt{N})C\d^{1-s} -(\textcolor{red}{??}) \geq \eta/3.\]
\end{enumerate}

Let $\g_{m}:[0,1] \to \Omega$ be a PL $c$-cigar curve with $\g_m(0)\in X\cap \D_{1}^{(m-1)}$ and $\g_m(1) \in X\cap \D^{(m-1)}_2$. By (\ref{eq:cigar}) and inductive assumption (P3), we have that for all $t\in [0,1]$, 
\begin{equation}\label{eq:gamma}
\dist(\g_m(t),\partial\Omega) \geq (2c)^{-1}\min\{\dist(\g_m(0),\partial\Omega),\dist(\g_m(1),\partial\Omega)\} 
%\geq (2c)^{-1} \min\{\dist(H_1,\partial\Omega),\e\} 
\geq (2c)^{-1}\eta.
\end{equation}
Using inductive assumption (P2), we can find a number $r_m>0$ such that
\begin{enumerate}
\item $r_m < \tfrac13\dist(\D^{(m-1)}_{1},\D^{(m-1)}_i)$ for all $i\in\{2,\dots,k_{m-1}\}$
\item $r_m < (6c)^{-1}\min\{\eta,\e\} -  \sum_{i=1}^{k_{m-1}} \diam{\D^{(m-1)}_{i}}$.
\end{enumerate}

The second property of $r_m$ implies that 
\[ r_m <(2c)^{-1}\dist(\D^{(m-1)}_{1},\partial\Omega) - \diam{\D^{(m-1)}_{1}}\]
which, along with (\ref{eq:gamma}), implies that $\mathcal{N}(\D_{1}^{(m-1)},\g_m,2r_m) \subset \Omega$. 
%\[ 0 < r_m < \min\left\{(2c)^{-1}\dist(H^{(m-1)}_{i},\partial\Omega) - \diam{H^{(m-1)}_{i}}, \, \min_{\substack{i,j \in \{1,\dots,k_{m-1}\}\\i\neq j}}\frac{\dist(H^{(m-1)}_{i},H^{(m-1)}_j)}{3}\right\}.\]
Let
\[T_m = \sup\left \{ t \in [0,1] : \mathcal{N}(\D_{1}^{(m-1)},\g_m|_{[0,t]},r_m) \cap \bigcup_{j =2}^{k_m-1} \D_j^{(m-1)} = \emptyset \right \}. \]
Since $\dist(\D_1^{(m-1)},\D_i^{(m-1)}) \geq 3r_m$ for all $i\neq 1$, we have that $T_m>0$. Let $i_0 \in \{2,\dots,k_{m-1}\}$ be such that 
\[ \mathcal{N}(\D_{1}^{(m-1)},\g_m|_{[0,T_m]},r_m) \cap \D_{i_0}^{(m-1)} \neq \emptyset.\]
For simplicity, we may assume that $i_0=2$. Denote by $H^{(m)}$ the bi-Lipschitz path given from Lemma \ref{lem:1} for the curve $\g = \g_m|_{[0,T_m]}$. Consider now the disjoint closed sets  
\[ E_1 = H_1^{(m)}(\D^{(m-1)}_{1})\cup \D^{(m-1)}_{2}, \quad E_{2} = \D^{(m-1)}_{3}, \quad \dots, \quad E_{k_m-1} = \D^{(m-1)}_{k_{m-1}}\]
and apply Lemma \ref{lem:convexsub} to the sets $E_j$ to obtain mutually disjoint closed convex sets $\D^{(m)}_1,\dots,\D^{(m)}_{k_{m}}$ with $k_{m} \leq k_{m-1}-1$. We note that 
\begin{enumerate}
\item for each $i\in \{1,\dots,k_{m}\}$ there exists $j\in\{1,\dots,n\}$ with $D_{j} \subset \D^{(m)}_{i}$;
\item \[ \sum_{i=1}^{k_{m}} \diam{\D^{(m)}_{i}} \leq  \sum_{i=1}^{k_{m-1}} \diam{\D^{(m-1)}_{i}} + r_m < (6c)^{-1}\min\{\eta,\e\}.\]
\end{enumerate}
It follows that $\D_1^{(m)},\dots,\D_{k_{m}}^{(m)}$ are contained in $\Omega$ and in fact, for each $i\in\{1,\dots,k_{m}\}$
\begin{align*} 
\dist(\D_i^{(2)},\partial\Omega) \geq \dist(\mathcal{T}_{\d}(X),\partial\Omega) - \diam{\D_i^{(m)}} &\geq \eta/2 - (6c)^{-1} > \eta/3.
\end{align*}
Therefore, we have verified that inductive assumptions (P1)--(P3) hold for $m$. If $k_{m} = 1$ the procedure terminates and we proceed to Part 2; otherwise proceed to the next step.

\medskip

After $p$ steps, for some $p \in \{ 0,\dots,n-1\}$, we have $k_{p}=1$. By the choice of $\d$ and numbers $r_1,\dots,r_p$, the final convex set $\D^{(p)}_1$ satisfies properties (P1)--(P3); precisely, we have
\begin{enumerate}
\item $\diam{\D^{(p)}_1} < (6c)^{-1}\min\{\eta,\e\}$;
\item there exists $i\in\{1,\dots,n\}$ such that $D_i \subset \D^{(p)}_1$,
\item $\D^{(p)}_1 \subset\Omega$ and $\dist(\D^{(p)}_1,\partial\Omega) > \eta/3$.
\end{enumerate}

%\begin{align*} 
%\diam{^{(p)}_1} \leq  \sum_{i=1}^{n} \diam{D_i} + (r_1+\cdots+ r_p) &\leq (1+5\sqrt{N})C\d^{1-s} + (r_1+\cdots+ r_p)\\
%&\leq (4c)^{-1}\e + p(12cn)^{-1}\e\\
%&\leq (3c)^{-1}\e.
%\end{align*}
%Moreover there exists $i\in\{1,\dots,n\}$ such that $D_i \subset H^{(p)}_1$. Therefore, $H^{(p)}_1 \subset\Omega$ and
%\begin{align*} 
%\dist(H^{(p)}_1,\partial\Omega) \geq (\eta - 5\sqrt{N}\d) - (1+5\sqrt{N})C\d^{1-s} -(r_1+\cdots+r_p).
%\end{align*}

\subsubsection{Part 2} 
Let $z_0\in \Omega$ be the center of $B$ and let $\g_{p+1} : [0,1] \to \Omega$ be a PL $c$-cigar curve in $\Omega$ with $\g_{p+1}(0) \in X \cap \D^{(p)}_1$ and $\g_{p+1}(1) =z_0$. If $z_0 \in X \cap \D^{(p)}_1$, then we can choose $\g_{p+1}$ to be constant. By (\ref{eq:cigar}), we have that for all $t\in [0,1]$, 
\begin{align}\label{eq:gamma2} 
\dist(\g_{p+1}(t),\partial\Omega) &\geq (2c)^{-1}\min\{\dist(\g_{p+1}(0),\partial\Omega),\dist(\g_{p+1}(1),\partial\Omega)\}\\
&\geq (2c)^{-1}\min\{\e,\eta\}. \notag%\\  
%&\geq (2c)^{-1} \min\{\dist(H^{(p)}_1,\partial\Omega),\e\}\\
%&\textcolor{red}{\geq 2r.}
\end{align}

Let $r_{m+1}$ be a positive number with $r_{m+1} < (6c)^{-1}\min\{\eta,\e\}$. Then (\ref{eq:gamma2}) implies that 
\[ \mathcal{N}(\D_{1}^{(p)},\g_{p+1},r_{p+1}) \subset \Omega.\] 
Let now $H^{(p+1)}$ be the bi-Lipschitz path given from Lemma \ref{lem:1} for $\g = \g_{p+1}$. If $p=0$, then we define $H: [0,1] \to LIP(\R^N)$ with $H = H^{(p+1)}$. If $p\geq 1$, we concatenate the bi-Lipschitz paths $H^{(1)},\dots,H^{(p+1)}$ and we obtain the desired bi-Lipschitz path $H$.

%\subsection{Bi-Lipschitz collapsing for polygonal domains} Note that the proof of Proposition \ref{prop:1} yields a slighlty stronger statement.
%
%\begin{corollary}\label{cor:collapse}
%Let $C>0$, $c\geq 1$, $s\in [0,1)$, $\eta,\e \in (0,1)$, and let $\d$ be the constant given by (\ref{eq:delta}). Let $\Omega \subset \R^2$ be a $c$-John domain, and let $E\subset \Omega$ so that 
%be a union of some closed set with finitely may components $E_1,\dots,E_k$ such that
%\begin{enumerate}
%\item $E$ is a union of squares in $\mathscr{G}_{\d}$,
%\item $\dist(E,\partial\Omega) \geq (\eta/2)\diam{\Omega}$
%\item $E$ has finitely many components $E_1,\dots,E_n$ and
%\[ \sum_{i=1}^n \diam{E_i} \leq (12c)^{-1}\min\{\e,\eta\}\diam{\Omega}.\]
%\end{enumerate}
%If $B\subset \Omega$ is a disk of radius $\e\diam{X}$, then there exists a bi-Lipschitz path $H:[0,1] \to LIP(\Omega)$ such that
%\begin{enumerate}
%\item $H_1$ maps $E$ into $B$;
%\item for each $t\in [0,1]$ and each $i\in\{1,\dots,n\}$, the map $H_t|E_i$ is an isometry. 
%\end{enumerate}
%\end{corollary} 

\section{A multitwist bi-Lipschitz map}\label{sec:multitwist}

In \textsection\ref{sec:proofprop} we prove Proposition \ref{prop:main} while in \textsection\ref{sec:f} we show that the multitwist map in Theorem \ref{thm:main} is bi-Lipschitz.

\subsection{Proof of Proposition \ref{prop:main}}\label{sec:proofprop}

In this subsection we prove Proposition \ref{prop:main}. To that end, we require the following ``egg-yolk principle" lemma which is a simple application of Koebe's Distortion Theorem. 

\begin{lemma}
\label{lem:koebe}
Given $\d>0$, there exists $L_0>1$ with the following property. If $U$ is a domain in $\R^2$, $K \subset U$ is a compact connected set with $\dist(K,\partial U) \geq \d \diam{K}$, $x_0\in K$ is a point, and $f:U \to \R^2$ is an injective conformal map, then for all $x,y \in K$,
\[ L_0^{-1}|f'(x_0)| |x-y|  \leq |f(x)-f(y)| \leq L_0|f'(x_0)| |x-y|.\]
\end{lemma}

\begin{proof}
If $K$ is a single point, the claim is trivial. Assume for the rest that $\diam{K} = d>0$. Let $V$ be a maximal $(\d d/4)$-separated subset of $K$ containing $x_0$. By the doubling property of $\R^2$, there exists $N\in\N$ depending only on $\d$ such that $\card{V} \leq N$. 

By the Koebe Distortion Theorem (see for example \cite[Theorem I.4.5]{GM} and \cite[Theorem 1.3]{Pomm}), there exists a universal $A>1$ such that for any $z\in K$ and for any $w,w_1,w_2\in B(z,\frac12\d d)$ we have
\begin{equation}\label{Koebe1}
A^{-1}|f'(w)| |w_1-w_2| \leq |f(w_1)-f(w_2)| \leq A|f'(w)| |w_1-z_2|
\end{equation}
\begin{equation}\label{Koebe2}
A^{-1}|f'(z)| \leq |f'(w)| \leq A |f'(z)|
\end{equation}
\begin{equation}\label{Koebe3}
\dist(f(z),\partial f(U)) \geq A^{-1} \d d |f'(z)|.
\end{equation}
By (\ref{Koebe2}), we have that for all $x\in K$, 
\begin{equation}\label{Koebe4}
A^{-N}|f'(x_0)| \leq |f'(x)| \leq A^N |f'(x_0)|.
\end{equation}

We show that $f|K$ is $(L_1|f'(x_0)|)$-Lipschitz for some $L_1>0$ depending only on $\d$. Fix $x,y \in K$ and consider two cases. If $|x-y|<\d d/2$, then by (\ref{Koebe1}) and (\ref{Koebe4})
\[ A^{-N}|f'(x_0)| |x-y|  \leq |f(x)-f(y)| \leq A^{N}|f'(x_0)| |x-y|.\]
Suppose now that $|x-y|\geq \d d/2$. Then, there exist $z,z'\in V$ such that $x\in B(z,\d d/4)$ and $y\in B(z',\d d/4)$, and by connectedness of $K$, there exist distinct $z_1,\dots,z_l \in V$ such that $z_1=z$, $z_l =z'$, and for all $j\in\{1,\dots,l-1\}$, $|z_j-z_{j+1}| < \d d/2$. Therefore,
\begin{align*}
|f(x)-f(y)| &\leq |f(x)-f(z)| + \sum_{i=1}^{l-1}|f(z_{i+1})-f(z_i)| + |f(x')-f(z')|\\ 
&\leq |f'(x_0)|(N+1)A^N (\d d/2)\\ 
&\leq |f'(x_0)|(N+1)A^N |x-y|.
\end{align*}

%To finish the proof, we show that $f^{-1}|f(K)$ is $L^2$-Lipschitz for some $L_2>0$ depending only on $\d$. 
By (\ref{Koebe3}) we have that $\dist(w,\partial f(U)) \geq A^{-N}\d d|f'(x_0)|$ for all $w\in f(K)$. On the other hand, since $f|K$ is $L_1$-Lipschitz, we have that $\diam{f(K)} \leq L_1|f'(x_0)| d$. Therefore, 
\[ \dist(f(K),\partial f(U)) \geq (L_1A^N)^{-1}\d\diam{f(K)}.\] 
Then, working as above, we can find $L_2>0$ depending only on $L_1$ and $N$ (hence only on $\d$) such that $f^{-1}|f(K)$ is $(L_2|(f^{-1})'(f(x_0))|)$-Lipschitz. Therefore, for all $x,y \in K$
\[ |x - y| \leq \frac{L_2}{|f'(x_0)|} |f(x)-f(y)| \]
and the proof is complete.
\end{proof}

We can now prove Proposition \ref{prop:main}.

\begin{proof}[{Proof of Proposition \ref{prop:main}}]
Let $X \subset \R^2$ be a $c$-uniformly disconnected set. By Theorem \ref{thm:UD-hyper} we know that there is a geodesic pants decomposition of the hyperbolic Riemann surface $S:= \S^2 \setminus X$ so that the cuffs $(\alpha_j)$ have uniformly bounded hyperbolic length. By Proposition \ref{prop:collar}, there exist mutually disjoint ring domains $(R_j')$ which are thickenings of $(\alpha_j)$ with a uniform upper bound $M_0$ on their moduli. 

For each $j$, denote by $V_j'$ and $U_j'$ the bounded and unbounded, respectively, components of $\R^2 \setminus R_j'$. Let $\zeta_j$ be a similarity of $\R^2$ such that $\diam{\zeta_j^{-1}(V_j')}=1$ and $0 \in \zeta_j^{-1}(V_j')$. By (\ref{eq:Loewner}), there exists $\e_0$ depending only on $M_0$ (hence only on $c$) such that $\dist(\partial\zeta_j^{-1}(U_j'),\partial\zeta_j^{-1}(V_j')) \geq \e_0$. By Lemma \ref{lem:MM}, there exists a polygonal Jordan curve $\g_j$ with edges in $\mathscr{G}_{\e_0/16}^1$ which encloses $\zeta_j^{-1}(V_j')$ and satisfies
\[ \e_0/16 \leq \dist(x,\zeta_j^{-1}(V_j')) \leq \e_0/2, \qquad\text{for all $x\in \g_j$}.\]
Applying Lemma \ref{lem:MM}, there exists a polygonal Jordan curve $\G_j$ with edges in $\mathscr{G}_{\e_0/32}^1$ which encloses $\g_j$ and satisfies
\[ \e_0/32 \leq \dist(x,\g_j) \leq \e_0/4, \qquad\text{for all $x\in \G_j$}.\]
The ring domain $R_j''$ bounded by $\g_j$ and $\G_j$ satisfies 
\begin{enumerate}
\item $\dist(\g_j,\G_j) \geq \e_0/32$,
\item $1\leq \diam{R_j''} \leq 1+ \frac32\e_0$ and 
\item $\dist(x, \partial\zeta_j(R_j')) \geq \e_0/16$, for all $x\in R_j''$.
\end{enumerate}
It follows that $R_j'' \subset [-1- \frac32\e_0,1+ \frac32\e_0]^2$ and since the boundary curves of $R_j''$ are made of edges in $\mathscr{G}_{\e_0/32}^1$, there are at most $k$ many different domains $R_j'' $, with $k$ depending only on $\e_0$, hence only on $c$. 

There exists $\d_0\in (0,1)$ depending only on $\e_0$ (hence only on $c$) and for each $j$ there exists $\d_j \in (0,1-\d_0)$, and there exists a conformal map 
\[ \psi_j : B(0,1) \setminus \overline{B}(0,\d_j) \to R_j'' .\] 
%Denote by $\G$ the circle $\partial B(0, \frac12(1+r_j))$. 
Setting
\[ K := \{ 1 - \tfrac34 \d_0 \leq |x| \leq 1 -\tfrac14 \d_0\} \subset  U:= B(0,1) \setminus \overline{B}(0,\d_j), \]
we have $\dist ( K, \partial U) \geq \d_0/4$ and $\diam{K} = 2-\d_0/2$. Hence by Lemma \ref{lem:koebe},
we have that $|\psi_j'(1 - \tfrac34 \d_0)|^{-1}\psi_j$ restricted on $K$ is a $L_0$-bi-Lipschitz, where $L_0$ depends only on $\d_0$ (hence only on $c$). Moreover,
\[ \frac1{L_0(2-\d_0/2)} \leq  \frac{\diam{\psi_j(K)}}{L_0\diam{K}} \leq  |\psi_j'(1 - \tfrac34 \d_0)| \leq L_0\frac{\diam{\psi_j(K)}}{\diam{K}} \leq L_0\frac{\sqrt{2}(2+3\e_0)}{2-\d_0/2}.\]

For each $j \in\N$, let $\lambda_{j} = \diam{\psi_j(\partial B(0, 1 - \frac34\d_0))} \in [1,1+3\e_0/2]$. It follows that the map 
\[ (\lambda_j)^{-1}\psi_j|K\] 
is $L_1$-bi-Lipschitz for some $L_1$ depending on $L_0,\e_0,\d_0$, hence only on $c$.

To complete the proof set 
\[ L = \max\left\{\frac{4-\d_0}{2\d_0}, \frac{4L_1}{4-\d_0} \right\},\]
define conformal maps
\[g_j: \overline{B(0,1)} \setminus B(0,1-1/L) \to \R^2 \qquad \text{with}\qquad g_j(x) = (\lambda_j)^{-1}\psi_j|K((1-\d_0/4)x),\]
and define similarities
\[ \phi_j : \R^2 \to \R^2 \qquad\text{with}\qquad \phi_j(x) = (\lambda_j)^{-1}\zeta_j^{-1}(x).\]
Since $L \geq \frac{4-\d_0}{2\d_0}$, we have that $(1-\d_0/4)x \in K$ for all $x \in \overline{B(0,1)} \setminus B(0,1-1/L)$. Moreover, since $L\geq  \frac{4L_1}{4-\d_0}$, we have that $g_j$ is $L$-bi-Lipschitz. Since there are at most $k$ many domains $R_j''$, there are at most $k$ many conformal maps $g_j$.
\end{proof}

Setting
\[ f_j = \phi_j \circ g_{i(j)} \qquad\text{and}\qquad R_j = f_j(\overline{B(0,1)} \setminus B(0,1-1/L))\]
where $\phi_j$ and $g_{i(j)}$ are as in the statement of Proposition \ref{prop:main}, and applying Lemma \ref{lem:koebe} to the ring 
\[ K' = \overline{B(0,1- \epsilon_0/8)} \setminus B(0,1-7\epsilon_0/8),\] 
we see that there exists $\xi >0$ so that
\begin{equation}
\label{eq:ringdist}
\dist ( \partial R_j' , R_j) \geq \dist ( \partial K_j , R_j) \geq \xi \diam{R_j}
\end{equation}
for all $j$.

\subsection{A multitwist bi-Lipschitz map}\label{sec:f}

For the rest of this section we fix a $c$-uniformly disconnected Cantor set $X \subset \R^2$. By Proposition \ref{prop:main}, we obtain $k\in\N$, $L>1$, a finite set $\{g_1,\dots,g_k\}$ of $L$-bi-Lipschitz conformal maps defined on $\overline{B}(0,1)\setminus B(0,1-\tfrac1{L})$, similarities $(\phi_j)_{j\in\N}$ and ring domains $R_j$ such that for each $j\in\N$ there exists $i(j)\in\{1,\dots,k\}$
\begin{equation} 
\label{eq:gj}
R_j = f_j( \overline{B}(0,1)\setminus B(0,1-\tfrac1{L}))\qquad \text{with}\qquad  f_j = \phi_j \circ g_{i(j)}.
\end{equation}
Let $f:\R^2 \to \R^2$ be a map such that $f$ is the identity outside of the union of $\overline{R_j}$, while for each $j\in\N$, $f|R_j = f_j \circ \mathfrak{D}\circ f_j^{-1}$ with
\[ \mathfrak{D}(r,\theta) = (r, \theta + 2\pi L(1-r)).\]

\begin{lemma}\label{lem:BLmap}
The map $f$ is $L_0$-bi-Lipschitz with $L_0$ depending only on $c$.
\end{lemma}

\begin{proof}
It is fairly elementary to see that $\mathfrak{D}$ is $L_1$-bi-Lipschitz for some $L_1>1$ depending only on $L$ (hence only on $c$).
%We first show that $D$ is $L_1$-bi-Lipschitz for some $L_1$ depending only on $L$ (hence only on $c$). Let $x = R e^{i\theta}$ and $y = re^{i\phi}$ with $r, R \in [1-\frac1{L} ,1]$ and $\theta,\phi \in [0,2\pi)$. Suppose first that $|R-r|\geq (4L)^{-1}|e^{i\theta}-e^{i\phi}|$. Then,
%\[ |D(R,\theta) - D(r,\phi)| \geq |R-r| \geq (1+4L)^{-1}|R e^{i\theta}-re^{i\phi}|\]
%while
%\begin{align*} 
%|D(R,\theta) - D(r,\phi)| &\leq |R-r| + r|e^{i\theta}e^{2\pi i L(R-r)} -e^{i\phi}| \\ 
%&\leq |R-r| + |e^{i\theta}-e^{i\phi}| + |e^{2\pi i L(R-r)} -1|\\
%&\leq (4L+4\pi L+1)|R-r|\\
%&\leq  (4L+4\pi L+1)|R e^{i\theta}-re^{i\phi}|.
%\end{align*}
%Suppose now that $|R-r|\leq (4 L)^{-1}|e^{i\theta}-e^{i\phi}|$. Then,
%\begin{align*}
%|D(R,\theta) - D(r,\phi)| & \leq |R-r| + |e^{i\theta}-e^{i\phi}| + |e^{2\pi i L(R-r)} -1|\\ 
%&\leq (4\pi L + 1)|R-r| + |e^{i\theta}-e^{i\phi}|\\ 
%&\leq (\pi+2)|e^{i\theta}-e^{i\phi}|\\ 
%&\leq (1-\tfrac1{L})^{-1}(\pi+2)|R e^{i\theta}-re^{i\phi}|
%\end{align*}
%while
It follows that for each $j\in\N$, $f|R_j$ is $L^2L_1$-bi-Lipschitz. Since $f$ is the identity outside of the union of $\overline{R_j}$ (and hence bi-Lipschitz), we get that $f$ is an $L_2$-bounded length distortion map for some $L_2>1$ depending only on $L$. That is, 
\[ L_2^{-1}\ell(\g) \leq \ell(f(\g)) \leq L_2\ell(\g)\]
for any rectifiable curve $\g$, with $\ell$ denoting length. The proof is completed by recalling that every bounded length distortion homeomorphism of $\R^2$ (or any quasiconvex space) is bi-Lipschitz quantitatively.
\end{proof}

\section{Decomposion and proof of Theorem \ref{thm:main}}\label{sec:proof}

In this section we will prove the following result, which immediately implies Theorem \ref{thm:main1}.

\begin{theorem}\label{thm:main}
Suppose the Assouad dimension of $X$ is less than $1$ and $f$ is the bi-Lipschitz map from \textsection\ref{sec:f}. Then there exists a bi-Lipschitz path $H: [0,1] \to LIP(\R^2)$ such that $H_0 = f$ and $H_1$ is the identity.
\end{theorem}

The proof comprises of 4 steps. In the first step we relabel the ring domains $R_j$ obtained from Proposition \ref{prop:main}. In the second step we use Proposition \ref{prop:1} to unwind the Dehn twists in each $R_j$ without changing small neighborhoods of $X$. In the third step we compose the bi-Lipschitz paths from the second step to perform unwindings arbitrarily close to $X$. Finally, in the fourth step, we use the uniformity of our maps to take a limit in the sequence of bi-Lipschitz paths obtained from the third step and recover the desired bi-Lipschitz path.

For the rest, we denote by $(R_j)_{j\in\N}$, $(\phi_j)_{j\in\N}$, $\{g_1,\dots,g_k\}$, and 
\[ (f_j)_{j\in\N} = (\phi_j\circ g_{i(j)})_{j\in\N}\] 
the ring domains, similarities, and conformal maps, respectively, from Proposition \ref{prop:main}.

%\begin{lemma}
%Set $D_j$ to be the bounded component of the complement of $R_j$ and $X_j$ to be $X\cap D_j$. By shrinking $R_j$ if necessary, we may assume that:
%\begin{enumerate}[(i)]
%\item there exists $M_0 > 0$ such that $M(R_j) \leq M_0$, 
%\item there exists $\epsilon >0$ such that $D_j$ contains a ball of radius $\epsilon \diam X_j$,
%\item there exists $c\geq 1$ such that $D_j$ is a $c$-John domain for all $j$,
%\item there exists $\eta >0$ such that $\dist (X_j , \partial D_j) \geq \eta \diam X_j$.
%\end{enumerate}
%\end{lemma}

\subsection{Step 1: Relabelling the ring domains $R_j$.} This step is similar to the proof of Proposition \ref{prop:UDsuff}.

For each $j\in \N$ let $V_j$ and $U_j$ be the bounded and unbounded, respectively, components of $\R^2 \setminus \overline{R_j}$. 

 Let $\varepsilon$ be the empty word. There exist three distinct $l_1,l_2,l_3 \in \N$ such that 
\begin{enumerate}
\item for all $j\in \N$, there exists $i\in\{1,2,3\}$ such that $R_j \subset V_{l_i}$ and
\item for all $j\in\N$ and all $i\in\{1,2,3\}$, $R_{l_i} \cap V_j = \emptyset$.
\end{enumerate}
For each $l\in\{1,2,3\}$, we denote $R_{l,\varepsilon}= R_{i_l}$ where $\varepsilon$ denotes the empty word. 

Inductively, suppose that for some $l\in\{1,2,3\}$ and for some finite word $w\in\{1,2\}^*$ we have labelled $R_{l,w} = R_{j_0}$ where $j_0 \in \N$. Then there exist exactly two distinct $j_1,j_2 \in \N$ such that
\begin{enumerate}
\item $R_{j_1},R_{j_2} \subset V_{j_0}$ and
\item for all $j\in\N\setminus \{j_1,j_2\}$ with $R_j \subset V_{j_0}$, either $R_j \subset V_{j_1}$, or $R_j \subset V_{j_2}$. 
\end{enumerate}
We denote $R_{l,w1}=R_{j_1}$ and $R_{l,w2} = R_{j_2}$.

Thus, we have that $\{R_j:j\in\N\} = \{R_{l,w} : l\in\{1,2,3\}, w\in\{1,2\}^*\}$. Given $ l\in\{1,2,3\}$ and $w\in\{1,2\}^*$ we denote by $V_{l,w}$ and $U_{l,w}$ the bounded and unbounded, respectively, components of $\R^2 \setminus \overline{R_{l,w}}$.
Further, denote by $X_{l,w}$ the intersection $X_{l,w} = X \cap V_{l,w}$.  

Moreover, if $R_j = R_{l,w}$ we set $\phi_{l,w} = \phi_j$ and $f_{l,w} = f_{j}$. In particular, $f_{l,w} = \phi_{l,w} \circ g_{i(l,w)}$.

By Proposition \ref{prop:main} we have that for all $ l\in\{1,2,3\}$, $w\in\{1,2\}^*$ and $i\in\{1,2\}$
\begin{equation}
\label{eq:chain} 
\frac{\diam{R_{l,wi}}}{\diam{R_{l,w}}} \leq \frac{\diam{V_{l,w}}}{\diam{R_{l,w}}} \leq \frac{\diam{R_{l,w}}-2\dist(V_{l,w},U_{l,w})}{\diam{R_{l,w}}} .
\end{equation}

Suppose $\dist ( V_{l,w},U_{l,w})$ is realized by $|x-y|$. Then since $x,y \in \partial R_{l,w}$ and $f$ is the identity there, we have by \eqref{eq:gj} that for some $j\in \N$,
\begin{align}\label{eq:ratio}
\dist(V_{l,w},U_{l,w}) &= |f(x) - f(y)| \geq \frac{ \diam {R_{l,w}} }{L} |(f_{l,w}^{-1})'(x)-(f_{l,w}^{-1})'(y)| \geq  \frac{ \diam{R_{l,w}}}{L^2}.
\end{align}
We conclude via \eqref{eq:chain} that 
\begin{equation}
\label{eq:shrink} 
\frac{\diam{R_{l,wi}}}{\diam{R_{l,w}}}\leq 1 - \frac{1}{L^2} .
\end{equation}

\subsection{Step 2: Unwinding the Dehn twist in $R_{l,w}$ while acting as isometries on neighbourhoods of $X_{l,w}$.}

For each $l\in\{1,2,3\}$ and $w\in \{1,2\}^*$ we define a bi-Lipschitz path $H_{l,w}:[0,1] \to LIP(\R^2)$ as follows. 

First, set $H_{l,w}| U_{l,w}$ to be the identity. Second, define $H_{l,w} | R_{l,w}$ so that for each $t\in[0,1]$
\[ (H_{l,w} | R_{l,w})_t =  f_{l,w} \circ D_{1-t} \circ (f_{l,w})^{-1}\]
recalling $D_t$ from Lemma \ref{lem:dehntwistpath}. 

\begin{lemma}
\label{lem:gjtwists}
The family of bi-Lipschitz paths 
\[ \mathcal{F} := \left\{ H_{l,w} | R_{l,w} : l\in\{1,2,3\}, w\in\{1,2\}^*\right\},\] 
which unwinds the Dehn twist in each $R_{l,w}$, is a uniform family of bi-Lipschitz paths.
\end{lemma}

\begin{proof}
For each $i \in\{1,\dots,k\}$ and each $t\in [0,1]$ set $H^i_t = g_i\circ D_{1-t} \circ g_i^{-1}$. By Proposition \ref{prop:path}, each $H^i$ is a bi-Lipschitz path. Now for each $i \in\{1,\dots,k\}$ let
\[ \mathcal{G}^i = \left\{\phi_{l,w} \circ H^i \circ (\phi_{l,w})^{-1} :l\in\{1,2,3\}, w\in\{1,2\}^*\right\}.\]
Since $X$ is bounded, there exists $c$ depending on the diameter of $X$ such that each $\phi_{l,w}$ has a scaling factor at most $c$. Therefore, by Lemma \ref{lem:pathdil}, $\mathcal{G}^i$ is a uniform family of bi-Lipschitz paths. Note that $\mathcal{F} \subset \bigcup_{i=1}^k \mathcal{G}^i$ so $\mathcal{F}$ is a uniform family of bi-Lipschitz paths as a finite union of uniform families of bi-Lipschitz paths.
\end{proof}

Before defining $H_{l,w}|V_{l,w}$ we make some remarks. 
%Below, for each $i \in \{1,\dots,k\}$ denote by $\texttt{V}_i$ the Jordan domain bounded by $g_i(\partial B(0,1-L^{-1}))$. We have that for each $ l\in\{1,2,3\}, w \in\{1,2\}^*\}$ there exists $i(l,w)\in\{1,\dots,k\}$
%\[ \{\texttt{V}_i : i=1,\dots,k\} = \{\phi^{-1}_{l,w}(V_{l,w}) : l\in\{1,2,3\}, w \in\{1,2\}^*\}.\]

First, there exist $C>0$ and $s \in (0,1)$ such that for any $l\in\{1,2,3\}$ and $w\in\{1,2\}^*$ the set $\phi_{l,w}^{-1}(X_{l,w})$ is $(C,s)$-homogeneous.

Second, since $\{\phi^{-1}_{l,w}(V_{l,w})\}_{l,w}$ is a finite collection of Jordan domains with smooth boundary, there exists $c>1$ such that for all $ l\in\{1,2,3\}$ and $w \in\{1,2\}^*$, the domain $\phi^{-1}_{l,w}(V_{l,w})$ satisfies (\ref{eq:cigar}) with constant $c$. 

Third, by the bi-Lipschitz Schoenflies Theorem \cite[Theorem A]{TukiaBL}, there exists $L'>1$ depending only on $L$ such that every $g_i$ extends to be an $L'$-bi-Lipschitz map on $\overline{B}(0,1)$. Therefore, for each  $ l\in\{1,2,3\}$ and $w \in\{1,2\}^*$ there exists a disk $B_{l,w} \subset \phi^{-1}_{l,w}(V_{l,w})$ such that 
\[ \text{radius}(B_{l,w}) \geq \e \quad\text{and}\quad \dist(B_{l,w} , \partial \phi_{l,w}^{-1}(V_{l,w})) \geq \e \]
with $\e:= (2L')^{-1}(1-L^{-1})$.

Fourth, by \eqref{eq:ringdist}, there exists $\eta>0$ such that for all $l\in\{1,2,3\}$ and $w\in\{1,2\}^*$
\[ \dist (\phi_{l,w}^{-1}(X_{l,w}) , \partial \phi_{l,w}^{-1}(V_{l,w})) \geq  \eta \diam{ \phi_{l,w}^{-1}(V_{l,w}) } \geq \eta .\]

Let $\d$ be the constant given in (\ref{eq:delta}) depending only on $C,s,c,\eta,\e$ above.
% Recall that there are at most $k$ different sets $\phi^{-1}_{l,w}(V_{l,w})$. 
Recall from the proof of Proposition \ref{prop:main} that for all $ l\in\{1,2,3\}$ and $w \in\{1,2\}^*$, $\phi_{l,w}^{-1}(V_{l,w})\subset [-1-\frac32\e_0,1+\frac32\e_0]^2$. Therefore, there exist at most $k_1$ different configurations for $\mathcal{T}_{\d}(\phi^{-1}_{l,w}(X_{l,w}))$ inside $\phi^{-1}_{l,w}(V_{l,w})$. Applying Proposition \ref{prop:1} for each of these finitely many cases we obtain bi-Lipschitz paths $\{H_1,\dots,H_{k_2}\}$ such that for each $l\in\{1,2,3\}$ and $w\in\{1,2\}^*$, there exists $j(l,w) \in\{1,\dots,k_2\}$ for which
\begin{enumerate}
\item $H_{j(l,w)}:[0,1] \to LIP(\phi_{l,w}^{-1}(V_{l,w}))$,
\item $H_{j(l,w)}$ is an isometry on each component of $\mathcal{T}_{\d}(\phi^{-1}_{l,w}(X_{l,w}))$,
\item $(H_{j(l,w)})_1$ maps $\mathcal{T}_{\d}(\phi^{-1}_{l,w}(X_{l,w}))$ onto $B_{l,w}$.
\end{enumerate}

By \eqref{eq:shrink}  there exists $p\in \N$, so that if $u\in \{1,2 \}^p$ then
\begin{equation*}
 R_{l,wu} \subset \phi_{l,w}\left(\mathcal{T}_{\d}(\phi^{-1}_{l,w}(X_{l,w})) \right).
\end{equation*}

We define $H_{l,w}|V_{l,w}$ as follows.
\begin{enumerate}[(a)]
\item For $0\leq t\leq 1/3$, we set 
\[ (H_{l,w}|V_{l,w})_t = \phi_{l,w}\circ (H_{i(l,w)})_{3t} \circ \phi^{-1}_{l,w}\] 
to be the path which moves $\phi_{l,w}\left(\mathcal{T}_{\d}(\phi^{-1}_{l,w}(X_{l,w}))\right)$ into the disk $\phi_{l,w}(B_{l,w})$.
\item For $2/3 \leq t \leq 1$, we set 
\[ (H_{l,w}|V_{l,w})_t = (H^l_w|V_{l,w})_{1-t}.\]
\item For $1/3 \leq t \leq 2/3$, we define $H_{l,w}|V_{l,w}$ as a path of rotations. Fix $l,w$ and suppose that $B_{l,w} = B(z_0,r)$. 
%and $z = se^{i\phi}\in B(z_0,r)$, we define
%\[ (H^l_w)_t(z) = z_0 + se^{i(\phi + 2\pi  (3t-1))} .\]
Find a conformal map 
\[ \psi_{l,w} : \phi^{-1}_{l,w}(V_{l,w}) \setminus \overline{ B_{l,w}} \to \{z : 1 < |z| < \rho_{l,w} \}\] 
for some $\rho_{l,w} >1$. Since the boundary $\phi^{-1}_{l,w}(V_{l,w})$ is smooth, $\psi_{l,w}$ extends smoothly on $\partial \phi^{-1}_{l,w}(V_{l,w})$. We apply Proposition \ref{prop:interp} with $P,Q$ given by
\begin{align*}
Q_t(z) &= \psi_{l,w}\circ \phi^{-1}_{l,w}\circ (H_{l,w}|\partial V_{l,w})(z)\\
P_t(z) &= \psi_{l,w}(z_0 + (z-z_0)e^{2\pi  (1-3t)i)}).
\end{align*}
Here $H_{l,w}|\partial V_{l,w}$ agrees with $H_{l,w}$ on the inner boundary component of $R_{l,w}$, recalling the construction in Lemma \ref{lem:gjtwists}.
This yields a bi-Lipschitz path 
\[P_{l,w}:[0,1] \to LIP(\{z : 1 \leq  |z| \leq \rho_{l,w} \}).\] 
By Proposition \ref{prop:path}, $G_{l,w} := (\psi^l_w)^{-1} \circ P^l_w \circ \psi^l_w$ is a bi-Lipschitz path. Since there are finitely many different pairs $(\phi^{-1}_{l,w}(V_{l,w}), B_{l,w})$, the set 
\[ \{G_{l,w} : l\in\{1,2,3\},w\in\{1,2\}^*\}\] 
is finite. Set now for $1/3 \leq t \leq 2/3$,
\[ (H_{l,w}|V_{l,w})_{t} = \phi_{l,w} \circ (G_{l,w})_{3t-1}\circ \phi^{-1}_{l,w}.\]
\end{enumerate}

By the finiteness of the family $\{G_{l,w}\}_{l,w}$, and working as in Lemma \ref{lem:gjtwists}, we see that $\{H_{l,w}|V_{l,w} : l\in\{1,2,3\},w\in\{1,2\}^*\}$ is a uniform family of bi-Lipschitz paths. 

By Lemma \ref{lem:partition}, $\{H_{l,w} : l\in\{1,2,3\},w\in\{1,2\}^*\}$ is a uniform family of bi-Lipschitz paths.
%By the uniformity of our construction, the family of all bi-Lipschitz paths $H^l_w$, for $l\in \{1,2,3\}$, $w\in \{1,2\}^*$ and $t\in [0,1]$, forms a uniform family of bi-Lipschitz paths, recalling Definition \ref{def:unifpath}.
The key point in the construction of $H_{l,w}$ is that it unwinds the Dehn twist in $R_{l,w}$ and acts as an isometry on $R_{l,{wu}}$ for any $u\in \{1,2\}^p$.

\subsection{Step 3: Composing unwindings in a controlled way.}
The next step is to combine the paths $H_{l,w}$ defined above. Let $k \in \{0,1,\ldots, p-1 \}$. Define 
\begin{equation*}
(F^k_0)_t(z) = 
\begin{cases}
(H_{l,w})_t(z) & \text{$z\in  R_{l,w} \cup V_{l,w}$ for $|w| = k$, $l \in \{1,2,3\}$,} \\
z & \text{otherwise.} \end{cases}
\end{equation*}
This is a bi-Lipschitz path.
For example, for $k=0$, this path unwinds the Dehn twists in the three outermost rings $R_{1,\epsilon}, R_{2,\epsilon}, R_{3,\epsilon}$. Then for $j\in \N$, suppose that $F^k_{j-1}$ has been defined. We then define
\begin{equation}
\label{eq:fkj}
(F^k_j)_t(z) = 
\begin{cases}
(F^k_{j-1})_t \circ ( H_{l,w})_t (z), & \text{$z\in R_{l,w}\cup V_{l,w}$, $|w| = k+jp$, $l\in \{1,2,3\}$,}\\
(F^k_{j-1})_t(z), & \text{otherwise.} \end{cases}
\end{equation}
If $|w| = k+jp$ and $l \in \{1,2,3\}$, then $(F^k_{j-1})_t$ acts as an isometry on $R_{l,w} \cup V_{l,w}$. 
Hence Lemma \ref{lem:pathcomp} implies that the composition in \eqref{eq:fkj} gives a bi-Lipschitz path, and we conclude that $F^k_j$ is a bi-Lipschitz path which unwinds the Dehn twists in $R_{l,w}$ for $|w| = k, k+p, k+2p,\ldots, k+jp$ and $l=1,2,3$.

\subsection{Step 4: Taking a limit.}
Set $F^k$ by $(F^k)_t = \lim_{j\to \infty} (F^k_j)_t$ for all $t\in [0,1]$. We claim that $F^k$ is a bi-Lipschitz path. To that end, first consider, for $n\in \N$, the domain
\[ \mathcal{U}_n := \bigcup_{\substack{|w| = k+np\\ l \in \{1,2,3\}}} U_{l,w}.\]
By construction, on this set we have $F^k|\mathcal{U}_n = F^k_n|\mathcal{U}_n$, and hence $F^k|\mathcal{U}_n$ is a bi-Lipschitz path. 

Next, note from \eqref{eq:fkj} that $F^k_j$ is obtained from $F^k_{j-1}$ by modifications from a uniform family of bi-Lipschitz paths (namely, the family $\{H_{l,w}\}_{l,w}$) on a region where $F^k_{j-1}$ acts as a family of isometries in a uniform way. By Lemma \ref{lem:pathcomp}, it follows that the family 
\[ \{ F^k_j|\bigcup_{n\in \N} \mathcal{U}_n : j\in \N \}\] 
is a uniform family of bi-Lipschitz paths. Hence $F^k|\bigcup_{n\in \N } \mathcal{U}_n$ is a bi-Lipschitz path.

Since $\bigcup_{n\in \N } \mathcal{U}_n = \R^2 \setminus X$, an application of Proposition \ref{prop:remov} shows that $F^k$ is in fact a bi-Lipschitz path on all of $\R^2$ which unwinds the Dehn twists in $R_{l,w}$ for $|w| \in k+p\N$, $l = 1,2,3$. Hence the concatenation of the finitely many paths $F^0, F^1, \cdots, F^{p-1}$ yields a bi-Lipschitz path which connects $f$ to the identity.

\section{A decomposable multitwist with singular set of large Assouad dimension}\label{sec:example}

Let $D_{\varepsilon}$ be the rectangle $[-\sqrt{2},\sqrt{2}]\times[-1,1]$, let $\alpha \in (0,1)$ and let
\begin{align*}
D_1 &= [-\sqrt{2}(1-\tfrac12\a), -\sqrt{2}\tfrac12\a] \times [\a-1,1-\a]\\
D_2 &= [\sqrt{2}\tfrac12\a, \sqrt{2}(1-\tfrac12\a)] \times [\a-1,1-\a].
\end{align*}
%Embed inside $D_0$ two copies $D_1,D_2$ of $D_0$ scaled down by a factor of $\frac1{\sqrt{2}}(1-\alpha)$ 
as in Figure \ref{fig:6}. Here $\varepsilon$ denotes the empty word. For each $i\in\{1,2\}$ let $\phi_i$ be the similarity of $\R^2$ mapping $D_{\varepsilon}$ onto $D_i$ with scaling factor $\frac1{\sqrt{2}}(1-\alpha)$. Let $X$ be the Cantor set attractor of the iterated function system $\{\phi_1,\phi_2\}$. 

\begin{figure}[h]
\begin{center}
\includegraphics[scale=6]{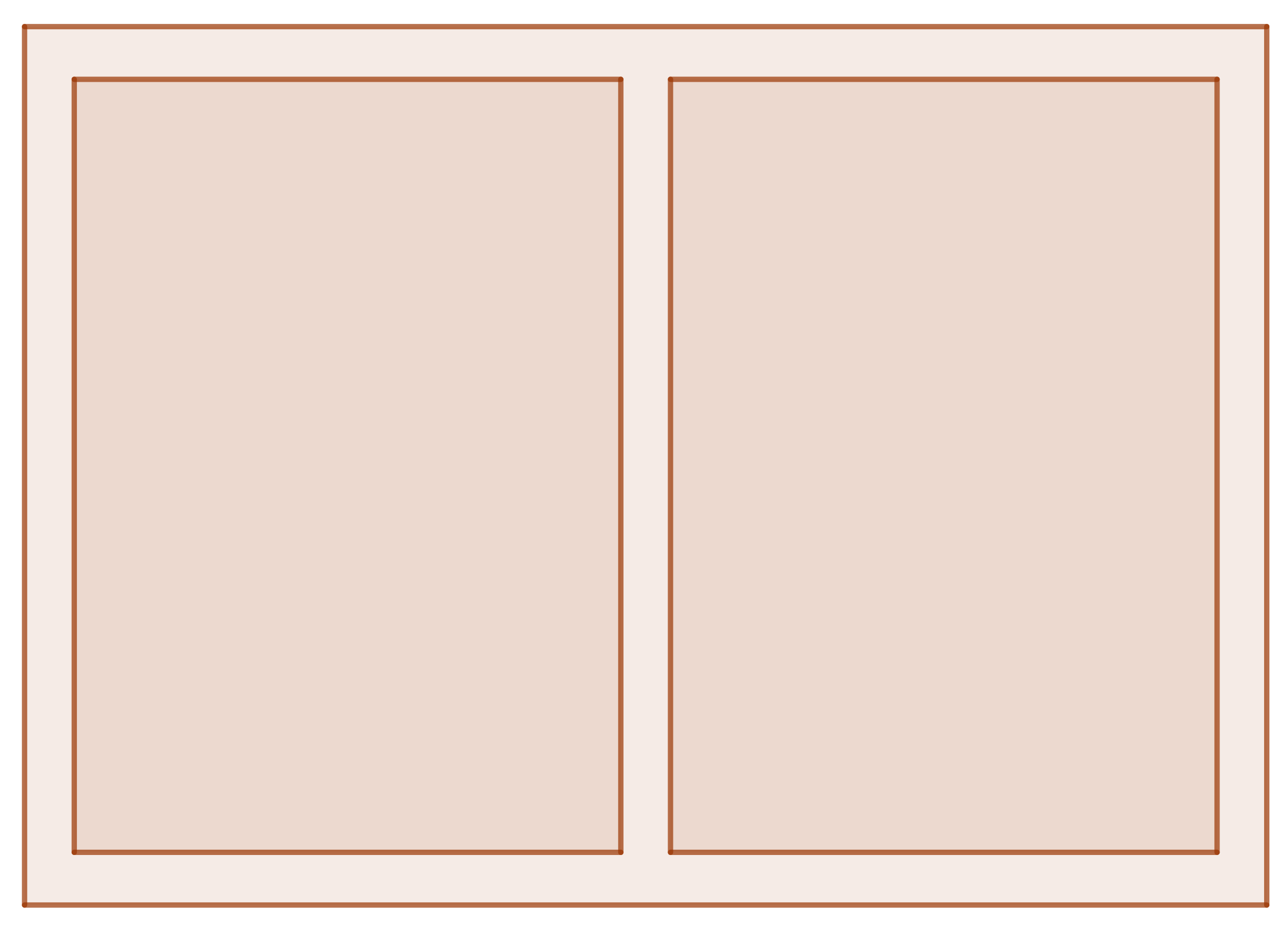}
\caption{The first two steps in the construction of $X$.}
\label{fig:6}
\end{center}
\end{figure}

By self-similarity, $X$ is uniformly disconnected and its Assouad dimension is 
\[ \dim_A(X) = \frac{\log{2}}{\log{\sqrt{2}}-\log(1-\alpha)}\]
which is greater than 1 when $\alpha$ is sufficiently small. Moreover, there exists a multitwist bi-Lipschitz map $f$ as in Section \ref{sec:f}, and by self-similarity, the set of maps $\{g_j\}$ in Proposition \ref{prop:main} contains one single element.

We claim that the map $f$ is decomposable. To prove the claim, we follow the arguments in Section \ref{sec:proof}. We may assume that the domains $\{V_{l,w}\}$ are exactly the interiors of the rectangles $\{D_w\}$. For simplicity, we drop the index $l$. The only step in the proof that we need to check (and the only one that requires the assumption on the Assouad dimension) is the existence of bi-Lipschitz paths $H_{j(w)}$. Since the collection $\{g_j\}$ contains only one element, we only need to construct for each $\e>0$ a ``collapsing'' bi-Lipschitz path $H: [0,1] \to D_{\varepsilon}$ which, for some small $\delta>0$, is an isometry on each component of $\mathcal{T}_{\delta}(X)$ and maps $\mathcal{T}_{\delta}(X)$ into a ball $B$ in $D_{\varepsilon}$ of radius $\e$.

We give a rough sketch of the construction of $H$ and leave the details to the reader. Fix $\e>0$. Choose $\beta\in (0,1)$ such that $(1-\a)(1+\beta)<1$ and choose $n\in\N$ such that 
\[  (1+\beta)^{n}(1-\a)^n< \tfrac14\e.\] 
The bi-Lipschitz path $H$ is a concatenation of $n$ bi-Lipschitz paths $H_1,\dots,H_n$. Let $H_1$ be the bi-Lipschitz path that is identity outside of $\bigcup_{w\in \{1,2\}^{n-1}}D_w$ and for each $w\in\{1,2\}^{n-1}$, it moves $D_{w1}$ towards $D_{w2}$ so that they both end up in a rectangle $D_w'$ with sides parallel to the axes and side-lengths 
\[ 4(1+\beta)(\tfrac1{\sqrt{2}}(1-\a))^{n}, \quad 2\sqrt{2}(1+\beta)(\tfrac1{\sqrt{2}}(1-\a))^{n}.\] 
The choice of $\beta$ ensures that $D_w'$ is contained in $D_w$. Moreover, $H_1$ acts as an isometry on $D_{wi}$ for all $wi\in\{1,2\}^{n}$.

Assume now that for some $m\in \{1,\dots,n-1\}$ we have defined the paths $H_1,\dots,H_m$ and assume that 
\begin{enumerate}
\item the concatenation of these paths  is the identity outside of $\bigcup_{w\in \{1,2\}^{n-m}} D_w$,
\item for each $w\in  \{1,2\}^{n-m}$, the concatenation has moved $X\cap D_w$ inside a rectangle $D_{w}' \subset D_w$ with sides parallel to the axes and side-lengths 
\[ 2\sqrt{2}(\sqrt{2})^{m}(1+\beta)^{m}(\tfrac1{\sqrt{2}}(1-\a))^{n}, \qquad 2(\sqrt{2})^{m}(1+\beta)^{m}(\tfrac1{\sqrt{2}}(1-\a))^{n},\]
\item for each $u\in\{1,2\}^n$  the concatenation of these paths acts as an isometry on $D_u$.
\end{enumerate}
Let $H_{m+1}$ be the bi-Lipschitz path that is identity outside of $\bigcup_{w\in \{1,2\}^{n-m-1}}D_w$ and for each $w\in\{1,2\}^{n-m-1}$, it moves $D_{w1}'$ towards $D_{w2}'$ so that they both end up in a rectangle $D_w'$ with sides parallel to the axes and side-lengths 
\[ 2\sqrt{2}(\sqrt{2})^{m+1}(1+\beta)^{m+1}(\tfrac1{\sqrt{2}}(1-\a))^{n}, \qquad 2(\sqrt{2})^{m+1}(1+\beta)^{m+1}(\tfrac1{\sqrt{2}}(1-\a))^{n}.\] 
Note that $H_{m+1}$ acts as an isometry on $D_{u}$ for all $u\in\{1,2\}^{n}$.

Finally, the concatenation $H$ of paths $H_1,\dots,H_n$ is the identity outside of $D$, acts as an isometry on $D_{u}$ for all $u\in\{1,2\}^{n}$, and $H(X)$ is contained in a rectangle $D' \subset D$ with side-lengths 
\[ 2\sqrt{2}(1+\beta)^{n}(1-\a)^n, \quad 2(1+\beta)^{n}(1-\a)^n.\] 
By the choice of $n$, the rectangle $D'$ has diameter less than $\e$ and the proof  is complete.

\bibliography{UD-bib}
\bibliographystyle{amsbeta}

\end{document}